\documentclass[psamsfonts]{amsart}

\usepackage{amssymb,amsfonts}
\usepackage[all,arc]{xy}
\usepackage{enumerate}
\usepackage{mathrsfs}
\usepackage{textcomp}
\usepackage{cite}
\usepackage{url}
\usepackage{mathtools}
\usepackage[font=small,labelfont=bf]{caption}

\newtheorem{thm}{Theorem}[section]
\newtheorem{cor}[thm]{Corollary}
\newtheorem{prop}[thm]{Proposition}
\newtheorem{lem}[thm]{Lemma}

\newtheorem{quest}[thm]{Question}

\theoremstyle{definition}
\newtheorem{defn}[thm]{Definition}

\newtheorem{exmp}[thm]{Example}

\newtheorem{fact}[thm]{Fact}
\newtheorem{condition}[thm]{Condition}

\theoremstyle{remark}
\newtheorem{rem}[thm]{Remark}

\newtheorem{claim}[thm]{Claim}

\makeatletter
\let\c@equation\c@thm
\makeatother
\numberwithin{equation}{section}

\def\Ind{\setbox0=\hbox{$x$}\kern\wd0\hbox to 0pt{\hss$\mid$\hss} \lower.9\ht0\hbox to 0pt{\hss$\smile$\hss}\kern\wd0} 

\def\Notind{\setbox0=\hbox{$x$}\kern\wd0\hbox to 0pt{\mathchardef \nn=12854\hss$\nn$\kern1.4\wd0\hss}\hbox to 0pt{\hss$\mid$\hss}\lower.9\ht0 \hbox to 0pt{\hss$\smile$\hss}\kern\wd0} 

\def\ind{\mathop{\mathpalette\Ind{}}} 

\def\nind{\mathop{\mathpalette\Notind{}}}

\title{Criteria for exact saturation and singular compactness}

\author{Itay Kaplan, Nicholas Ramsey, and Saharon Shelah}

\date{\today}

\thanks{This work was supported by the European Research Council grant 338821. Paper no. 1192 in Shelah's  publication list. The first author would like to thank the Israel Science Foundation for its support of this research (grants no. 1533/14 and 1254/18)}

\begin{document}

\maketitle

\begin{abstract}
We introduce the class of \emph{unshreddable theories}, which contains the simple and NIP theories, and prove that such theories have exactly saturated models in singular cardinals, satisfying certain set-theoretic hypotheses.  We also give criteria for a theory to have singular compactness.  
\end{abstract}

\setcounter{tocdepth}{1}
\tableofcontents

\section{Introduction}

The construction of saturated models of a theory $T$ is sensitive to the combinatorial properties of sets definable in $T$.  Consequently, properties of saturated models and their constructions are often reflected in model-theoretic dividing lines, defined in terms of synactic properties of a formula.  For example, it is well known that a stable theory has a saturated model in every cardinal in which it is stable \cite[Theorem III.3.12]{shelah1990classification}. In a similar vein, the third-named author characterized the simple theories in terms of the \emph{saturation spectrum} of a theory, namely, the set of cardinal pairs $(\lambda, \kappa)$ with $\lambda \geq \kappa$ and every model of size $\lambda$ extends to a $\kappa$-saturated model of the same size \cite[Theorem 4.10]{shelah1980simple}.  Subsequent work on transferring saturation, Keisler's order, and the interpretability order all suggest that comparisons between saturated models and their constructions yield meaningful measures of model-theoretic complexity \cite{baldwin1999transfering, dvzamonja2004maximality, MR3666452}.  

A theory $T$ is said to have \emph{exact saturation} at the cardinal $\kappa$ if there is a $\kappa$-saturated model of $T$ which is not $\kappa^{+}$-saturated.  If $\kappa$ is regular and $>|T|$, \emph{every} theory has models with exact saturation at $\kappa$ \cite[Theorem 2.4, Fact 2.5]{20-Kaplan2015}, but for singular $\kappa$, this property connects with notions from classification theory.  The simplest example of a theory without exact saturation at singular $\kappa$ is the theory of dense linear orders.  Given a singular cardinal $\kappa$ and a $\kappa$-saturated dense linear order $I$ and given any subsets $A < B$ from $I$ with $|A| = |B| = \kappa$, there are cofinal and coinitial subsets $A_{0}$ and $B_{0}$ of $A$ and $B$ respectively with $|A_{0}| = |B_{0}|  < \kappa$.  It follows from the $\kappa$-saturation of $I$ that there is some $c \in I$ with $a < c < b$ for all $a \in A_{0}$ and $b \in B_{0}$, hence for all $a \in A$ and $b \in B$.  By quantifier elimination for the theory of dense linear orders, it follows that $I$ is $\kappa^{+}$-saturated.  This example suggests that failures of exact saturation are related to the presence of orders.  Indeed, it was shown in \cite[Theorem 4.10]{20-Kaplan2015} that an NIP theory $T$ has exact saturation at a singular cardinal $\kappa$ if and only if $T$ is not distal (assuming $2^{\kappa} = \kappa^{+}$ and $\kappa > |T|$).  

Additionally, \cite[Theorem 3.3]{20-Kaplan2015} showed that if $T$ is simple then $T$ has exactly $\mu$-saturated models for singular $\mu$ of cofinality greater than $|T|$ (again assuming $2^{\mu} = \mu^{+}$ and, additionally, $\square_{\mu}$).  In the unstable case, this argument started from a witness $\varphi(x;y)$ to the independence property along an indiscernible sequence $I$ of length $\kappa$ and inductively constructed a model $M$ containing $I$ so that every type over fewer than $\mu$ parameters is realized and also so that, for every tuple $c$ from $M$, there is an interval from the indiscernible sequence that is indiscernible over $c$.  This ensures that the model is both $\mu$-saturated yet omits the type $\{\varphi(x;a_{i})^{i \text{ even}} : i \in I\}$.  Simplicity theory, via the independence theorem and the forking calculus, played an important role in that argument.  

Here, we are interested in both finding criteria for exact saturation in broader model-theoretic contexts but also understanding the reach of the argument of \cite{20-Kaplan2015}, which was tailored to simple theories.  We introduce \emph{shredding}, a notion that refines forking and exactly captures the obstacle to ensuring that one can realize a formula such that a large interval of a given indiscernible sequence is additionally indiscernible over the realization.  This notion is defined with exact saturation in mind, but it appears to be a fairly fundamental notion and may have uses beyond the context explored here. We use shredding to define the class of \emph{unshreddable theories}, which are roughly the theories with a bound on the number of times a type can shred, and observe that both NIP and simple theories are unshreddable.  Our main theorem is that one may construct exactly saturated models of unshreddable theories with the independence property for singular cardinals satisfying certain set-theoretic hypotheses.  We follow the rough outline of the argument of \cite{20-Kaplan2015} but, in contrast to the approach taken there, which faced considerable technical issues in adapting the tools of simplicity theory for the construction of an exactly saturated model, our proof, in addition to being more general, is considerably simpler and more direct.  

In section 4, we focus on the way that the class of unshreddable theories compares to other classes from classification theory.  We show that there is an unshreddable theory with SOP$_{3}$, which suggests that the class of unshreddable theories is substantially broader than the simple theories.  However, we show subsequently that neither NSOP$_{1}$ nor NTP$_{2}$ imply that a theory is unshreddable.   

In section 5, we consider the dual problem of which conditions on a theory imply the inability to construct exactly saturated models, which we call \emph{singular compactness}.  We formulate one such criterion and show that this condition entails a considerable amount of complexity:  theories that meet our condition for every formula has TP$_{2}$ and SOP$_{n}$ for all $n$.  Nonetheless, we show that our condition restricted to a fixed finite set of formulas implies a local version of singular compactness. For this local variant, we show that there is an example which satisfies the condition for a fixed finite set of formulas which is NSOP$_{4}$.     

\section{Shredding}

\subsection{Basic definitions}

From now on, $T$ will denote a complete first-order theory with monster model $\mathbb{M}$.  Our model-theoretic notation and terminology is standard.  Following standard model-theoretic usage, we say the $A$-indiscernible sequence $I$ is \emph{extracted} from $J$ if $I$ realizes the EM-type of $I$ over $A$.  The existence of such a sequence follows by Ramsey and compactness.  In this subsection, we will describe shredding and show that it can be given a finitary characterization.  

\begin{defn}
Let $A$ be a set of parameters and $\lambda$ an infinite cardinal.
\begin{enumerate}
\item We say that $\varphi(x;a)$ $\lambda$-\emph{shreds} over $A$ when there is $\overline{b}$ such that:
\begin{enumerate}
\item $\overline{b} = \langle b_{\alpha} : \alpha < \lambda \rangle$ is an indiscernible sequence over $A$.
\item For no $\alpha < \lambda$ and $c \in \varphi(\mathbb{M},a)$ is $\overline{b}_{\geq \alpha}$ an indiscernible sequence over $Ac$.  
\end{enumerate}
\item We say a type $\lambda$-shreds over $A$ if it implies a formula that $\lambda$-shreds over $A$, respectively.  
\item We say $p \in S(B)$ $\lambda$\emph{-shreds} over $A$ with \emph{a built-in witness} if $A \subseteq B$ and an indiscernible sequence witnessing $\lambda$-shredding is contained in $B$.
\item For the above notions, we may omit $\lambda$ when $\lambda = (|T|+|A|)^{+}$.  
\item  
We define $\kappa^{m}_{\mathrm{shred}}(T)$ to be the minimal regular cardinal $\kappa$ such that there is no increasing continuous sequence of models $\langle M_{i} : i \leq \kappa \rangle$ and $p \in S^{m}(M_{\kappa})$ so that $p \upharpoonright M_{i+1}$ shreds over $M_{i}$ with a built-in witness, if such a cardinal exists (where \emph{continuous} means $M_{\delta} = \bigcup_{i < \delta} M_{i}$ for limit $\delta$).  Otherwise, we set $\kappa^{m}_{\mathrm{shred}}(T) = \infty$.    The cardinal $\kappa_{\mathrm{shred}}(T) = \sup_{m} \kappa^{m}_{\mathrm{shred}}(T)$.  
\item We say $T$ is \emph{unshreddable} if $\kappa_{\mathrm{shred}}(T) < \infty$.   
\end{enumerate}
\end{defn}

\begin{rem}
Though we do not use it, it is natural to additionally introduce an associated notion of forking:  say $\varphi(x;a)$ $\lambda$\emph{-shred-forks} over $A$ if $\varphi(x;a) \vdash \bigvee_{i < k} \psi_{i}(x;a_{i})$ where each $\psi_{i}(x;a_{i})$ $\lambda$-shreds over $A$.  This satisfies extension, by the same argument as for forking.  Note that, if $\varphi(x;a)$ $\lambda$-shreds over $A$, then, unless $\varphi(x;a)$ is inconsistent, we know $a$ is not contained in $A$.  
\end{rem}

The following lemma gives a finitary equivalent to $\lambda$-shredding.  

\begin{lem} \label{finitization}
Assume $\lambda = \text{cf}(\lambda) > |T| + |A|$.  The following are equivalent:
\begin{enumerate}
\item The formula $\varphi(x;a)$ $\lambda$-shreds over $A$.
\item There are $n$, $\overline{b}$, $\overline{\eta}$, and $\overline{\psi}$ satisfying:
\begin{enumerate}
\item $\overline{b} = \langle b_{\alpha} : \alpha < \lambda \rangle$ is an $A$-indiscernible sequence.
\item $\overline{\eta} = \langle \eta_{i} : i < k \rangle$ is a finite sequence of increasing functions in ${^{n}}(2n)$.
\item $\overline{\psi} = \langle \psi_{l}(x;y_{0},\ldots, y_{n-1};a'_{l}) : l < k \rangle$ is a sequence of formulas with $a'_{l} \in A$.
\item For every $\delta < \lambda$ divisible by $2n$ (or just for every limit $\delta < \lambda$), we have 
$$
\varphi(x;a) \vdash \bigvee_{l < k} \left[ \psi_{l}(x;b_{\delta}, \ldots, b_{\delta+n-1},a'_{l}) \leftrightarrow \neg \psi_{l}(x;b_{\delta + \eta_{l}(0)}, \ldots, b_{\delta + \eta_{l}(n-1)},a'_{l}) \right].
$$
\end{enumerate}
\end{enumerate}
\end{lem}

\begin{proof}
(2)$\implies$(1) is clear by definition of $\lambda$-shredding.  

(1)$\implies$(2).  Suppose $\varphi(x;a)$ $\lambda$-shreds over $A$ witnessed by the indiscernible sequence $\overline{b} = \langle b_{\alpha} : \alpha < \lambda \rangle$.  Then for each $\delta < \lambda$ consider the set of formulas $\Gamma_{\delta}(x)$ containing $\varphi(x;a)$ and every formula of the form  
$$
 \chi(x;b_{\delta}, \ldots, b_{\delta + m-1}) \leftrightarrow \chi(x;b_{\delta + \nu(0)}, \ldots, b_{\delta + \nu(m-1)}) 
$$
for every $m < \omega, \chi \in L(A)$, and increasing function $\nu \in {^{m}\lambda}$.  Note that if $c \models \Gamma_{\delta}(x)$, then $b_{\geq \delta}$ is $Ac$-indiscernible so $\Gamma_{\delta}(x)$ is inconsistent for all $\delta < \lambda$ by the definition of $\lambda$-shredding.  It follows by compactness that, for each $\delta < \lambda$, there is a finite sequence $\overline{\chi}^{\delta} = \langle \chi^{\delta}_{l}(x;y_{\delta}) : l < k_{\delta} \rangle$ with each $\chi^{\delta}_{l}(x;y_{\delta}) \in L(A)$, and (after adding dummy variables to ensure all formulas in $\overline{\chi}$ have the same parameter variables) there are $m_{\delta} < \omega$ and a sequence of increasing functions $\overline{\nu}_{\delta} = \langle \nu_{\delta,l} : l < k_{\delta} \rangle$ from ${^{m_{\delta}}\lambda}$ such that 
$$
\varphi(x;a) \vdash \bigvee_{l < k_{\delta}} \chi^{\delta}_{l}(x;b_{\delta}, \ldots, b_{\delta + m_{\delta}-1}) \leftrightarrow \neg \chi^{\delta}_{l}(x;b_{\delta + \nu_{\delta,l}(0)}, \ldots, b_{\delta + \nu_{\delta,l}(m_{\delta}-1)}).
$$
Let $u_{\delta} = \{i : i < m_{\delta}\} \cup \{\nu_{\delta,l}(i) : i < m_{\delta}, l <k_{\delta}\}$.  Let $n_{\delta}$ be the least natural number such that $|u_{\delta}| < n_{\delta}$.  

By the pigeonhole principle and the regularity of $\lambda$, there is a subset of limit ordinals $X \subseteq \lambda$ of size $\lambda$, $n,m < \omega$ and $\overline{\chi} = \langle \chi_{l} : l < k \rangle$ so that $\delta \in X$ implies $n_{\delta} = n$, $k_{\delta} = k$, $m_{\delta} = m$, and $\overline{\chi}^{\delta} = \overline{\chi}$.  Further refining $X$, we may assume $\delta < \delta'$ from $X$ implies $\delta + i < \delta'$ for all $i \in u_{\delta}$.  Let $Y = \{\delta + i : \delta \in X, i \in u_{\delta}\} \subseteq \lambda$.  Let $\langle \alpha_{i} : i < \lambda \rangle$ be an increasing enumeration of a subset of $\lambda$ containing $Y$ so that $\langle \alpha_{(2n) \cdot i} : i < \lambda \rangle$ enumerates $X$ (which is possible by the choice of $n$).  Then if $\delta = \alpha_{(2n)\cdot j} \in X$, we can find for each $l < k$ an increasing function $\eta_{\delta,l} \in {^{n}(2n)}$ so that
$$
\delta + \nu_{\delta,l}(i) = \alpha_{(2n)\cdot j+\eta_{\delta,l}(i)},
$$
for all $i < m$ (we do not place any constraints on $\eta_{\delta,l}(i)$ for $m \leq i < n$ other than the requirement that $\eta_{\delta, l}$ is an increasing function\textemdash note that $\alpha_{(2n)\cdot j + \eta_{\delta, l}(i)} < \alpha_{(2n)(j+1)}$, which is the next ordinal in $X$ after $\delta$).   Write $\overline{\eta}_{\delta}$ for this sequence of functions.  By one last application of the pigeonhole principle, we can find $X' \subseteq X$ of size $\lambda$ and $\overline{\eta}$ so that $\delta \in X'$ implies $\overline{\eta} = \overline{\eta_{\delta}}$ and let $\langle \alpha'_{i} : i < \lambda \rangle$ be an increasing enumeration of $\{\alpha_{i+j} : \alpha_{i} \in X',j < 2n\}$. Write $\overline{b}' = \langle b'_{i} : i < \lambda \rangle$ for the subsequence of $\overline{b}$ defined by $b'_{i} = b_{\alpha'_{i}}$.  Note that if $\delta$ is divisible by $2n$, then $\alpha'_{\delta} \in X'$.  

Unravelling definitions, we see that
$$
\varphi(x;a) \vdash \bigvee_{l < k} \chi_{l}(x;b'_{\delta},\ldots, b'_{\delta + m-1}) \leftrightarrow \neg \chi_{l}(x;b'_{\delta + \eta_{l}(0)},\ldots, b'_{\delta + \eta_{l}(m-1)}),
$$   
for all $\delta < \lambda$ divisible by $2n$.  Because $m < n$, by adding dummy variables to each $\chi_{l}$, we obtain formulas $\psi_{l}$ so that 
$$
\varphi(x;a) \vdash \bigvee_{l < k} \psi_{l}(x;b'_{\delta},\ldots, b'_{\delta + n-1}) \leftrightarrow \neg \psi_{l}(x;b'_{\delta + \eta_{l}(0)},\ldots, b'_{\delta + \eta_{l}(n-1)}),
$$
as desired.  
\end{proof}

\begin{rem}
The proof shows, in fact, that any sequence witnessing that $\varphi(x;a)$ $\lambda$-shreds over $A$ gives rise to a sequence $\overline{b}$ as in (2) by restricting to a subsequence.  
\end{rem}

\begin{cor} \label{indiscernible witness}
Assume $\lambda = \text{cf}(\lambda) > |T| + |A|$ and $\varphi(x;a)$ $\lambda$-shreds over $A$.  Then there is an $A$-indiscernible sequence $\langle b_{\alpha} : \alpha < \lambda \rangle$ and $m < \omega$ so that 
\begin{itemize}
\item $\langle (b_{m \cdot \alpha}, b_{m \cdot \alpha + 1}, \ldots, b_{m \cdot \alpha + m-1}) : \alpha < \lambda \rangle$ is $Aa$-indiscernible.  
\item $\langle b_{\alpha} : \alpha < \lambda \rangle$ witnesses that $\varphi(x;a)$ $\lambda$-shreds over $A$ and, additionally, for every $c \in \varphi(\mathbb{M},a)$ and $\alpha < \lambda$, the finite sequence $(b_{m \cdot \alpha}, b_{m \cdot \alpha + 1}, \ldots, b_{m \cdot \alpha + m-1})$ is not $Ac$-indiscernible.  
\end{itemize}
\end{cor}

\begin{proof}
Suppose $\varphi(x;a)$ $\lambda$-shreds over $A$.  By Lemma \ref{finitization}, there is an $A$-indiscernible sequence $\langle c_{\alpha} : \alpha < \lambda \rangle$, a number $n < \omega$, a sequence of $L(A)$-formulas $\overline{\psi} = \langle \psi_{l}(x;y_{0},\ldots, y_{n-1}) : l < k \rangle$, and a sequence $\overline{\eta} = \langle \eta_{l} : l < k \rangle$ with each $\eta_{l} \in {^{n}(2n)}$ an increasing function, such that, for every $\delta < \lambda$ divisible by $2n$, 
$$
\varphi(x;a) \vdash \bigvee_{l < k} \left[\psi_{l}(x;c_{\delta},\ldots, c_{\delta + n-1}) \leftrightarrow \neg \psi_{l}(x;c_{\delta + \eta_{l}(0)}, \ldots, c_{\delta + \eta_{l}(n-1)})\right].
$$
Let $m = 2n$ and extract an $Aa$-indiscernible sequence $\langle (b_{m \cdot \alpha}, b_{m \cdot \alpha + 1}, \ldots, b_{m \cdot \alpha + m-1}) : \alpha < \lambda \rangle$ from $\langle (c_{m \cdot \alpha}, c_{m \cdot \alpha + 1},\ldots, c_{m \cdot \alpha + m-1}) : \alpha < \lambda \rangle$.  Then for all $\delta < \lambda$ divisible by $2n$,
$$
\varphi(x;a) \vdash \bigvee_{l < k} \left[\psi_{l}(x;b_{\delta},\ldots, b_{\delta + n-1}) \leftrightarrow \neg \psi_{l}(x;b_{\delta + \eta_{l}(0)}, \ldots, b_{\delta + \eta_{l}(n-1)})\right].
$$
and $\langle b_{\alpha} : \alpha < \lambda \rangle$ is an $A$-indiscernible sequence, so we are done.  
\end{proof}

%
%

From Lemma \ref{finitization}, we obtain a variant of shredding that is somewhat more cumbersome and less natural, but will be useful in the arguments below.

\begin{defn}
For an infinite cardinal $\lambda$, we say $\varphi(x;a)$ \emph{explicitly} $\lambda$\emph{-shreds} over $A$ if there are $n$, $\overline{b}$, $\overline{\eta}$, and $\overline{\psi}$ satisfying:
\begin{enumerate}
\item $\overline{b} = \langle b_{\alpha} : \alpha < \lambda \rangle$ is an $A$-indiscernible sequence.
\item $\overline{\eta} = \langle \eta_{l} : l < k \rangle$ is a finite sequence of increasing functions in ${^{n}}(2n)$.
\item $\overline{\psi} = \langle \psi_{l}(x;y_{0},\ldots, y_{n-1};a'_{l}) : l < k \rangle$ is a sequence of formulas with $a'_{l} \in A$.
\item For every $\delta < \lambda$ divisible by $2n$, we have 
$$
\varphi(x;a) \vdash \bigvee_{l < k} \left[ \psi_{l}(x;b_{\delta}, \ldots, b_{\delta+n-1},a'_{l}) \leftrightarrow \neg \psi_{l}(x;b_{\delta + \eta_{l}(0)}, \ldots, b_{\delta + \eta_{l}(n-1)},a'_{l}) \right].
$$
\end{enumerate}
We will often say that the tuple $(\overline{b},n,\overline{\eta},\overline{\psi})$ witnesses that $\varphi(x;a)$ explicitly $\lambda$-shreds over $A$.  We say $\varphi(x;a)$ \emph{explicitly shreds} over $A$ if it explicitly $\lambda$-shreds over $A$ for some $\lambda$.  As before, we will say that a type $p$ over $B \supseteq A$ explicitly shreds over $A$ if it implies some formula that does, and it explicitly shreds over $A$ with a built-in witness if the witnessing $A$-indiscernible sequence $\overline{b}$ may be chosen to be contained in $B$.  
\end{defn}

The point of introducing this definition is that explicit shredding is a notion that lends itself to compactness arguments, as in the following easy lemma:

\begin{lem} \label{shredding equivalents}
The following are equivalent:
\begin{enumerate}
\item The formula $\varphi(x;a)$ shreds over $A$.
\item The formula $\varphi(x;a)$ explicitly shreds over $A$.
\item The formula $\varphi(x;a)$ explicitly $\aleph_{0}$-shreds over $A$.
\item The formula $\varphi(x;a)$ explicitly $\lambda$-shreds over $A$ for all infinite cardinals $\lambda$. 
\end{enumerate}
\end{lem}

\begin{proof}
(1)$\implies$(2) is Lemma \ref{finitization}, and (2)$\implies$(3) is immediate, by restricting the witnessing indiscernible sequence to an initial segment of length $\omega$.  (4)$\implies$(1) is also immediate, taking any $\lambda \geq (|A| + |T|)^{+}$, since explicit shredding implies shredding.  

(3)$\implies$(4)  Let $\lambda$ be any uncountable cardinal and suppose $\varphi(x;a)$ explicitly $\aleph_{0}$-shreds, witnessed by $(\overline{b},n,\overline{\eta},\overline{\psi})$, where $\overline{b} = \langle b_{i} : i < \omega \rangle$.  Define $b'_{i} = (b_{2n \cdot i}, \ldots, b_{2n \cdot i + 2n - 1})$ for all $i < \omega$.  The sequence $\langle b'_{i} : i < \omega \rangle$ is also $A$-indiscernible and, without loss of generality, by (the proof of) Corollary \ref{indiscernible witness}, we may assume further that it is $Aa$-indiscernible.  Then applying compactness, we can stretch it to $\overline{b}' = \langle b'_{i} : i < \lambda \rangle $ with $b'_{i} = (b_{2n \cdot i}, \ldots, b_{2n \cdot i + 2n - 1})$ for all $i < \lambda$.  Then the sequence $\langle b_{i} : i < \lambda \rangle$ is $A$-indiscernible and, together with $n$, $\overline{\eta}$, and $\overline{\psi}$ witnesses that $\varphi(x;a)$ explicitly $\lambda$-shreds.  This shows $(4)$.  
\end{proof}

\begin{lem} \label{no parameters}
Suppose $A$ is a set of parameters and $B \subseteq A$.  The following are equivalent:
\begin{enumerate}
\item $\varphi(x;a)$ shreds over $A$.
\item There is an $A$-indiscernible sequence $\overline{b} = \langle b_{i} : i < \lambda \rangle$ for $\lambda = (|A| + |T|)^{+}$ such that for no $c \in \varphi(\mathbb{M};a)$ and for no $\alpha < \lambda$ is $\overline{b}_{\geq \alpha}$ indiscernible over $Bc$.  
\item $\varphi(x;a)$ explicitly shreds over $A$ witnessed by a tuple $(\overline{b},n,\overline{\eta},\overline{\psi})$, where the formulas $\overline{\psi}$ have no parameters (i.e. are over the empty set).  
\end{enumerate}
\end{lem}

\begin{proof}
(2)$\implies$(1) is clear by the definition of shredding, since in particular (2) entails that for no $c \in \varphi(\mathbb{M};a)$ and $\alpha < \lambda$ is $\overline{b}_{\geq \alpha}$ indiscernible over $Ac$.  

(3)$\implies$(2) since, if $\overline{b} = \langle b_{\alpha} : \alpha < \lambda \rangle$, then, for all $\delta < \lambda$ divisible by $2n$, we have the implication
$$
\varphi(x;a) \vdash \bigvee_{l < k} \left[ \psi_{l}(x;b_{\delta}, \ldots, b_{\delta+n-1}) \leftrightarrow \neg \psi_{l}(x;b_{\delta + \eta_{l}(0)}, \ldots, b_{\delta + \eta_{l}(n-1)}) \right].
$$
which implies that no end segment of $\overline{b}$ can be indiscernible over a realization of $\varphi(x;a)$ (with no additional parameters).  \emph{A fortiori}, no end segment of $\overline{b}$ can be indiscernible over a set consisting of $B$ and a realization of $\varphi(x;a)$.  

To prove (1)$\implies$(3), we know, by Lemma \ref{shredding equivalents}, $\varphi(x;a)$ explicitly shreds over $A$, witnessed by the tuple $(\overline{b},n,\overline{\eta},\overline{\psi})$.  Let $c$ be a tuple enumerating the parameters occuring in the $\overline{\psi}$ and let $\overline{d}$ be the sequence $\overline{d} = \langle d_{\alpha} : \alpha < \lambda \rangle =  \langle (b_{\alpha},c) : \alpha < \lambda \rangle$, which is $A$-indiscernible since $\overline{b}$ was assumed to be $A$-indiscernible and $c$ comes from $A$.  Then it is easily seen that by merely adding dummy variables to the formulas $\overline{\psi}$, we get $\overline{\psi}' = \langle \psi'_{l} : l < k \rangle$ such that for every $\delta < \lambda$ divisible by $2n$, we have 
$$
\varphi(x;a) \vdash \bigvee_{l < k} \left[ \psi'_{l}(x;d_{\delta}, \ldots, d_{\delta+n-1}) \leftrightarrow \neg \psi'_{l}(x;d_{\delta + \eta_{l}(0)}, \ldots, d_{\delta + \eta_{l}(n-1)}) \right].
$$
Then $\varphi(x;a)$ explicitly shreds, witnessed by the tuple $(\overline{d},n,\overline{\eta},\overline{\psi}')$, where the formulas $\overline{\psi}'$ have no parameters.
\end{proof}

The direction (1)$\implies$(2) of Lemma \ref{no parameters} gives base monotonicity for shredding:

\begin{cor} \label{base monotonicity}
Suppose $B \subseteq A$ and $\varphi(x;a)$ shreds over $A$, then $\varphi(x;a)$ shreds over $B$.  
\end{cor}

\begin{prop} \label{the main equivalence prop}
Suppose $\kappa$ is a regular cardinal and $m < \omega$.  The following are equivalent:
\begin{enumerate}
\item There is an increasing sequence $\overline{A} = \langle A_{i} : i \leq \kappa \rangle$ with $A_{\kappa} = \bigcup_{i < \kappa} A_{i}$ and $p \in S^{m}(A_{\kappa})$ such that $p$ (explicitly) shreds over $A_{i}$ for all $i < \kappa$. 
\item There is an increasing continuous sequence of models $\overline{M} = \langle M_{i} : i \leq \kappa \rangle$ with $M_{\kappa} = \bigcup_{i < \kappa} M_{i}$ and some $p \in S^{m}(M_{\kappa})$ such that $p \upharpoonright M_{i+1}$ shreds over $M_{i}$ with a built-in witness.  
\end{enumerate}
\end{prop}

\begin{proof}
The direction (2)$\implies$(1) is immediate by Lemma \ref{shredding equivalents}, taking $A_{i} = M_{i}$ for all $i \leq \kappa$.  

(1)$\implies$(2):  for each $i < \kappa$, fix a formula $\varphi_{i}(x;a_{i}) \in p$ that explicitly shreds over $A_{i}$, witnessed by $(\overline{b}_{i}, n_{i}, \overline{\eta}_{i},\overline{\psi}_{i})$.  By Lemma \ref{no parameters}, we may assume that $\overline{b}_{i}$ and $\overline{\psi}_{i}$ have been chosen so that the formulas in $\overline{\psi}_{i}$ have no parameters.  By the regularity of $\kappa$, after replacing the sequence with a subsequence, we may assume $\varphi_{i}(x;a_{i}) \in p \upharpoonright A_{i+1}$.  Moreover, without loss of generality, we may assume $\overline{b}_{i} = \langle b_{i,j} : j < \omega \rangle$ for all $i < \kappa$.  

Our assumption that $\overline{\psi}_{i}$ contains no parameters entails that $\varphi_{i}(x;a_{i})$ explicitly shreds over any subset of $A_{i}$ and, in particular, that $\varphi_{i}(x;a_{i})$ shreds over $a_{<i}$.  Therefore we may replace $A_{i}$ by $a_{<i}$ and $p$ by $p \upharpoonright a_{<\kappa}$ and, hence, without loss of generality, the sequence $\langle A_{i} : i < \kappa \rangle$ is increasing and continuous.  

Let $\overline{\lambda} = \langle \lambda_{i} : i < \kappa \rangle$ be an increasing and continuous sequence of cardinals $\geq |T|$ with $\lambda_{i} \geq |A_{i}|$ and $\lambda_{i+1}$ regular for all $i < \kappa$.  Denote $\lim_{i < \kappa} \lambda_{i}$ by $\mu$.  Let $\overline{y} = \langle y_{j} : j < \mu \rangle$ be a sequence of variables of length $\mu$ and denote by $\overline{y}_{i}$ the restriction $\langle y_{j} : j < \lambda_{i} \rangle$ to the first $\lambda_{i}$ variables.  

Let $\Gamma(\overline{y},\overline{z}_{i} : i < \kappa)$ be a partial type over $A_{\kappa}$ such that the variables $\overline{z}_{i} = \langle z_{i,j} : j < \lambda_{i+1}\rangle$ have length $\lambda_{i+1}$, and which naturally expresses the following, for all $i < \kappa$:
\begin{enumerate}
\item The sequence $\overline{y}_{i}$ enumerates a model containing $A_{i}$.
\item The sequence $\overline{z}_{i}$ is indiscernible over $\overline{y}_{i}$, realizes the same EM-type over $A_{i}$ as $\overline{b}_{i}$, and is contained in $\overline{y}_{i+1}$.
\item The formula $\varphi_{i}(x;a_{i})$ explicitly shreds over $\overline{y}_{i}$, witnessed by $(\overline{z}_{i}, n_{i},\overline{\eta}_{i},\overline{\psi}_{i})$.  
\end{enumerate}
It suffices to show that this partial type is consistent, as to conclude we may take any complete type over the union of models realizing the $\overline{y}_{i}$ containing $\{\varphi(x;a_{i}) : i < \kappa\}$.  By compactness, it suffices to show this for $\kappa$ finite.  By induction on $\kappa < \omega$, we will show that we can find models and sequences satisfying the conditions in the partial type above.  Suppose this has been shown for $\kappa = l $.  By induction, we know there are models $\langle M_{j} : j < l \rangle$ and sequences $\langle \overline{c}_{j} : j < l\rangle$ satisfying the requirements.  Choose an arbitrary model $M$ of size $\lambda_{l}$ containing $A_{l}M_{l-1}\overline{c}_{l-1}$.  Extract an $M$-indiscernible sequence $\overline{b}'_{l}$ from $\overline{b}_{l}$.  Then $\overline{b}'_{l} \equiv_{A_{l}} \overline{b}_{l}$ so there is an automorphism $\sigma \in \mathrm{Aut}(\mathbb{M}/A_{l})$ with $\sigma(\overline{b}'_{l}) = \overline{b}_{l}$.  For each $j < l$, define $M'_{j} = \sigma(M_{j})$ and $\overline{c}'_{j} = \sigma(\overline{c}_{j})$, and then put $M'_{l} = \sigma(M)$.  

Finally, let $m = n_{l}$ and consider the sequence $\langle (b_{l,2m \cdot i}, \ldots, b_{l,2m \cdot i + 2m - 1}) : i < \omega \rangle$.  Let $\overline{b}''_{l} = \langle (b''_{2m \cdot i},\ldots, b''_{2m \cdot i + 2m - 1} ): i < \lambda_{l+1} \rangle$ be an $M'_{l}a_{l}$-indiscernible sequence realizing the same EM-type over $M_{l-1}A_{l}a_{l}$ as $\langle (b_{l,2m \cdot i}, \ldots, b_{l,2m \cdot i + 2m - 1}) : i < \omega \rangle$.  Then defining $\overline{c}'_{l} = \langle b''_{i} : i < \lambda_{l+1} \rangle$, we have that $\overline{c}'_{l}$ is an $M'_{l}$-indiscernible sequence and $\varphi_{l}(x;a_{l})$ explicitly shreds over $M'_{l}$, witnessed by $(\overline{c}'_{l}, n_{l}, \overline{\eta}_{l}, \overline{\psi}_{l})$.  It follows that $\langle M'_{j} : j < l+1 \rangle$ and $\langle \overline{c}'_{j} : j < l+1 \rangle$ satisfy the requirements, completing the induction and the proof. \end{proof}

\begin{rem} \label{regular size}
Note that, in the course of the proof Proposition \ref{the main equivalence prop}, we were able to replace each $A_{i}$ with $a_{<i}$, in which case we clearly have $|A_{i}| + \aleph_{0} = |i| + \aleph_{0}$ (in fact, for finite $i$ we have $|A_{i}| = l(a_{0})i$ and for infinite $i$ we have $|A_{i}| = |i|$).  

It follows, then, that if $\kappa$ is a regular cardinal and $\kappa^{m}_{\mathrm{shred}}(T) \geq \kappa^{+}$, then we can find a witness of the form $\langle M_{i} : i \leq \kappa \rangle$ and $p \in S^{m}(M_{\kappa})$ with $|M_{0}|$ an arbitrary regular cardinal $\geq |T|$, $\langle |M_{i}| : i < \kappa \rangle$ an increasing and continuous sequence of cardinals, and with $|M_{i+1}|$ a regular cardinal for all $i < \kappa$.  
\end{rem}

\subsection{Shredding and classification theory}

Here we establish some preliminary connections between the concepts of shredding and unshreddable theories with NIP and simplicity.  

\begin{defn}
Recall that the formula $\varphi(x;y)$ has the \emph{independence property} if for every $n$, there are $a_{0}, \ldots, a_{n-1}$ and tuples $b_{w}$ for every $w \subseteq \{0,\ldots, n-1\}$ so that 
$$
\models \varphi(a_{i},b_{w}) \iff i \in w.
$$
A theory is said to have the independence property if some formula does modulo $T$, otherwise T is NIP.  
\end{defn}

Equivalently, the formula $\varphi(x;y)$ has the independence property if there is an indiscernible sequence $\langle a_{i} : i < \omega\rangle$ and $b$ so that $\models \varphi(a_{i},b)$ if and only if $i$ is even (see, e.g., \cite[Lemma 2.7]{simon2015guide}).  

\begin{prop} \label{NIP implies unshreddable}
If $\lambda = \text{cf}(\lambda) > |T| + |A|$ and some consistent formula $\varphi(x;a)$ $\lambda$-shreds over $A$, then $T$ has the independence property.
\end{prop}

\begin{proof}
Suppose $\varphi(x;a)$ $\lambda$-shreds over $A$.  Then by Lemma \ref{shredding equivalents}, it explicitly $\lambda$-shreds so we may fix $k$, $n$, $\overline{\psi}$, $\overline{\eta}$, and $\overline{b} = \langle b_{\alpha} : \alpha < \lambda \rangle$ as in the definition of explicit shredding.  Let $c$ be an arbitrary element of $\varphi(\mathbb{M};a)$.  By the pigeonhole principle, there is a subset $X \subseteq \lambda$ of size $\lambda$, $l < k$, and $t \in \{0,1\}$ so that
$$
\models \psi_{l}(c;b_{\omega \cdot \alpha},\ldots, b_{\omega \cdot \alpha + n-1},a'_{l})^{t} \wedge \psi_{l}(c;b_{\omega \cdot \alpha + \eta_{l}(0)}, \ldots, b_{\omega \cdot \alpha + \eta_{l}(n-1)},a'_{l})^{1-t}
$$
for all $\alpha \in X$.  Let $\langle \alpha_{i} : i < \lambda \rangle$ be an increasing enumeration of $X$.  For $i < \lambda$ even, we define $d_{i} = (b_{\omega \cdot \alpha_{i}},\ldots, b_{\omega \cdot \alpha_{i} + n-1})$ and for $i < \lambda$ odd, we define $d_{i} = (b_{\omega \cdot \alpha_{i} + \eta_{l}(0)}, \ldots, b_{\omega \cdot \alpha_{i} + \eta_{l}(n-1)})$.  Then $\langle d_{i} : i < \lambda \rangle$ is an $A$-indiscernible sequence, by the $A$-indiscernibility of $\overline{b}$, and we have 
$$
c \models \{\psi_{l}(x,d_{i},a'_{l})^{t} : i < \lambda \text{ even}\} \cup \{\psi_{l}(x;d_{i},a'_{l})^{1-t} : i < \lambda \text{ odd}\},
$$
which shows $\chi(x,z;y) = \psi_{l}(x,y,z)$ has the independence property.  
\end{proof}

Recall that a formula $\varphi(x;a_{0})$ \emph{divides} over a set $A$ if there is an $A$-indiscernible sequence $\langle a_{i} : i < \omega \rangle$ such that $\{\varphi(x;a_{i}) : i < \omega\}$ is inconsistent.  A formula $\varphi(x;b)$ \emph{forks} over $A$ if $\varphi(x;b) \vdash \bigvee_{i < \kappa} \psi(x;a_{i})$ where each $\psi_{i}(x;a_{i})$ divides over $A$.  A type divides or forks over $A$ if it implies a formula that respectively divides or forks over $A$.  A theory is called \emph{simple} if there is a cardinal $\kappa$ such that, whenever $p$ is a type (in finitely many variables) over $A$, there is $B \subseteq A$ over which $p$ does not fork with $|B| < \kappa$.  The least such cardinal $\kappa$ is called $\kappa(T)$ and the least such regular cardinal is called $\kappa_{\mathbf{r}}(T)$.  

\begin{prop} \label{8-dividing implies forking}
If $\varphi(x;a)$ shreds over $A$ then $\varphi(x;a)$ forks over $A$.  
\end{prop}

\begin{proof}
Suppose $\lambda = (|T| + |A|)^{+}$ and, by Lemma \ref{shredding equivalents}, we know $\varphi(x;a)$ explicitly $\lambda$-shreds over $A$.  Hence, there are is an $A$-indiscernible sequence $\overline{b} = \langle b_{i} : i < \lambda \rangle$ such that that there is a sequence of $L(A)$-formulas $\langle \psi_{l}(x;y_{0},\ldots, y_{n-1}) : l < k \rangle$ and a sequence $\langle \eta_{l} : l < k \rangle$ with the property that that 
\begin{itemize}
\item [$(*)$] $\varphi(x;a) \vdash \bigvee_{l < k} \psi_{l}(x;b_{\delta},\ldots, b_{\delta + n-1}) \leftrightarrow \neg \psi_{l}(x;b_{\delta + \eta_{l}(0)},\ldots, b_{\delta + \eta_{l}(n-1)})$
\end{itemize}
for all $\delta < \lambda$ divisible by $2n$.  Given $\alpha < \lambda$, let $\overline{b}_{\alpha} = \langle b_{\omega \cdot \alpha + i} : i < \omega \rangle$.  By the proof of Corollary \ref{indiscernible witness}, we can moreover assume that $\langle \overline{b}_{\alpha} : \alpha < \lambda \rangle$ is an $Aa$-indiscernible sequence.  We will choose $(a_{\alpha})_{\alpha < \lambda}$ so that 
\begin{enumerate}
\item For all $\alpha < \lambda$, $a_{\alpha} \models \text{tp}(a/A\overline{b}_{<\alpha})$.
\item For all $\alpha < \lambda$, $\overline{b}_{\alpha}$ is an $a_{\alpha}A$-indiscernible sequence.
\end{enumerate}
Given $(a_{\beta})_{\beta < \alpha}$, to choose $a_{\alpha}$, first apply Ramsey and compactness to extract from $\overline{b}_{\alpha}$ a sequence $\overline{b}^{*}_{\alpha} = \langle b_{\omega \cdot \alpha + i}^{*} : i < \omega \rangle$ which is $Aa\overline{b}_{<\alpha}$-indiscernible.  Then as $\overline{b}_{\alpha} \equiv_{A\overline{b}_{<\alpha}} \overline{b}^{*}_{\alpha}$, we can choose $a_{\alpha}$ so that $a_{\alpha} \overline{b}_{\alpha} \equiv_{A\overline{b}_{<\alpha}} a \overline{b}^{*}_{\alpha}$.  The sequence $(a_{\alpha})_{\alpha < \lambda}$ satisfies both (1) and (2) by construction.  By Ramsey, compactness, and automorphism, we may moreover assume the sequence $\langle (a_{\alpha},\overline{b}_{\alpha}) : \alpha < \lambda \rangle$ is an $A$-indiscernible sequence.

By the finite Ramsey theorem, there is $n_{*}$ so that $n_{*} \to (2n)^{n}_{2^{k}}$.  Let $\Lambda = \{ \nu \in {^{2n}(n_{*})} : \nu \text{ increasing}\}$ and for $\nu \in \Lambda$, let $\overline{b}_{\alpha,\nu} = (b_{\omega \cdot \alpha + \nu(i)})_{i < 2n}$.  Let $\varphi'(x;\overline{b}_{\alpha,\nu})$ (suppressing parameters from $A$) denote the formula 
$$
\bigwedge_{l < k} \psi_{l}(x;b_{\omega \cdot \alpha + \nu(0)},\ldots, b_{\omega \cdot \alpha + \nu(n-1)}) \leftrightarrow \psi_{l}(x;b_{\omega \cdot \alpha + \nu(\eta_{l}(0))},\ldots, b_{\omega \cdot \alpha + \nu(\eta_{l}(n-1))}).
$$
Let $\varphi_{*}(x;a_{\alpha},\overline{b}_{\alpha,\nu})$ denote the formula $\varphi(x;a_{\alpha}) \wedge \varphi'(x;\overline{b}_{\alpha,\nu})$.  
\\
\\
\noindent \textbf{Claim 1}:  $\varphi(x;a_{0}) \vdash \bigvee_{\nu \in \Lambda} \varphi_{*}(x;a_{0},\overline{b}_{0,\nu})$.
\\
\\
\emph{Proof of claim}:  This proof is purely combinatorial and will not make use of (1), (2), or (*).  Let $c$ be any tuple with $\mathbb{M} \models \varphi(c;a_{0})$.  Given any increasing $\xi \in {^{n}(n_{*})}$, define $\chi(\xi) = \{l < k : \mathbb{M} \models \psi_{l}(c;b_{\xi(0)},\ldots, b_{\xi(n-1)},a'_{l})\}$.  This defines a coloring with $2^{k}$ possible colors.  As $n_{*} \to (2n)^{n}_{2^{k}}$, there is $\nu \in \Lambda$ so that $\nu$ is an increasing enumeration of a homogeneous subset of $n_{*}$ of size $2n$.  For each $l < k$, by homogeneity, both $(\nu(0),\ldots, \nu(n-1))$ and $(\nu(\eta_{l}(0)),\ldots, \nu(\eta_{l}(n-1)))$ take on the same value with respect to the coloring $\chi$, hence 
$$
\mathbb{M} \models \bigwedge_{l < k} \psi_{l}(c;b_{\nu(0)},\ldots, b_{\nu(n-1)}) \leftrightarrow \psi_{l}(c;b_{\nu(\eta_{l}(0))},\ldots, b_{\nu(\eta_{l}(n-1))}).
$$
This shows $\mathbb{M} \models \varphi_{*}(x;a_{0},\overline{b}_{0,\nu})$, proving the claim. \qed
\\
\\
\noindent \textbf{Claim 2}:  For each $\nu \in \Lambda$, $\varphi_{*}(x;a_{0},\overline{b}_{0,\nu})$ divides over $A$.  
\\
\\
\emph{Proof of claim}:  Let $\nu_{*} = (0,\ldots, 2n-1)$.  We will first show that $\varphi_{*}(x;a_{0},\overline{b}_{0,\nu_{*}})$ divides over $A$.  By $(*)$,
$$
\varphi(x;a) \vdash \neg \bigwedge_{l < k} \psi_{l}(x;b_{\omega \cdot \alpha},\ldots, b_{\omega \cdot \alpha + n-1}) \leftrightarrow \psi_{l}(x;b_{\omega \cdot \alpha + \eta_{l}(0)},\ldots, b_{\omega \cdot \alpha + \eta_{l}(n-1)}),
$$
and therefore $\varphi(x;a) \vdash \neg \varphi'(x;\overline{b}_{\alpha,\nu_{*}})$ for all $\alpha < \lambda$.  For all $\alpha$, we have $a_{\alpha} \equiv_{A\overline{b}_{<\alpha}} a$, so if $\beta < \alpha$, then $\varphi(x;a_{\alpha}) \vdash \neg \varphi'(x;\overline{b}_{\beta,\nu_{*}})$.  Therefore, when $\beta < \alpha$, we have 
$$
\varphi_{*}(x;a_{\alpha},\overline{b}_{\alpha,\nu_{*}}) \vdash \neg \varphi_{*}(x;a_{\beta},\overline{b}_{\beta,\nu_{*}}),
$$
from which it follows that $\{\varphi_{*}(x;a_{\alpha},\overline{b}_{\alpha,\nu_{*}}) : \alpha < \lambda\}$ is $2$-inconsistent.  Since $\langle (a_{\alpha},\overline{b}_{\alpha,\nu_{*}}) : \alpha < \lambda \rangle$ is an $A$-indiscernible sequence, we have shown $\varphi_{*}(x;a_{0},\overline{b}_{0,\nu_{*}})$ divides over $A$.  

Finally, as $\overline{b}_{0}$ is an $Aa_{0}$-indiscernible sequence, we have $\overline{b}_{0,\nu} \equiv_{Aa_{0}} \overline{b}_{0,\nu_{*}}$ for all $\nu \in \Lambda$.  It follows that $\varphi_{*}(x;a_{0},\overline{b}_{0,\nu})$ divides over $A$ for all $\nu \in \Lambda$.  This proves the claim and therefore proves the proposition, by Claim 1.  
\end{proof}

As a corollary, we obtain the following:

\begin{prop} \label{simple implies unshreddable}
If $T$ is simple, then $\kappa_{\mathrm{shred}}(T) \leq \kappa_{\mathbf{r}}(T)$.  
\end{prop}

\begin{proof}
Suppose not, so $\kappa_{\mathbf{r}}(T) < \kappa_{\mathrm{shred}}(T)$.  Let $\kappa = \text{cf}(\kappa) \geq \kappa_{\mathbf{r}}(T)$ with $\kappa < \kappa_{\mathrm{shred}}(T)$.  Then we have the following:
\begin{itemize} 
\item $\langle M_{i} : i \leq \kappa \rangle$ is an increasing sequence of models of $T$.
\item $p(x) = \{\varphi(x;a_{i}) : i < \kappa\}$ is a consistent partial type.
\item $\varphi(x;a_{i})$ shreds over $M_{i}$.
\item $a_{i}\in M_{i+1}$.
\end{itemize}
Then by Proposition \ref{8-dividing implies forking}, $p$ forks over $M_{i}$ for all $i < \kappa$.  Let $M_{\kappa} = \bigcup_{i < \kappa} M_{i}$.  As $T$ is simple, there is subset $A \subseteq M_{\kappa}$ with $|A| < \kappa_{\mathbf{r}}(T)$ such that $p$ does not fork over $A$.  As $\kappa$ is regular, there is some $i < \kappa$ so that $A \subseteq M_{i}$, from which it follows that $p$ does not fork over $M_{i}$ as well, a contradiction to the definition of $\kappa_{\mathbf{r}}(T)$.  
\end{proof}

\begin{cor} \label{contains NIP and simple}
The class of unshreddable theories contains the NIP and simple theories.  
\end{cor}

\begin{proof}
This follows immediately from Proposition \ref{NIP implies unshreddable} and Proposition \ref{simple implies unshreddable}.  
\end{proof}


\section{Respect and exact saturation}

\subsection{Respect} \label{respect subsection}

For the entirety of this subsection, we fix a singular cardinal $\mu$.  Writing $\text{cf}(\mu) = \kappa$, we will assume there is an increasing and continuous sequence of cardinals $\overline{\lambda} = \langle \lambda_{i} : i \leq \kappa \rangle$ such that $\lambda_{0} > \kappa$, $\lambda_{i+1}$ is regular for all $i < \kappa$, and $\lambda_{\kappa} = \mu$.  We will assume we have fixed for each $i < \kappa$ a sequence $\overline{a}_{i} = \langle a_{i,j} : j < \lambda_{i+1} \rangle$, which is $\overline{a}_{<i}$-indiscernible.  Additionally, we will assume that $T$ is a theory with $\kappa^{1}_{\mathrm{shred}}(T) \leq \kappa $.

\begin{defn} \label{respectdef}
Suppose $i < \kappa$ and $A$ is a set of parameters. 
\begin{enumerate}
\item  We say that $A$ \emph{respects} $\overline{a}_{i}$ when for any finite subset $C \subseteq A$, there is $\alpha < \lambda_{i+1}$ such that $\overline{a}_{i,\geq \alpha}$ is $C$-indiscernible. 
\item We say $p \in S^{<\omega}(A)$ \emph{respects} $\overline{a}_{i}$ when, for every $c \models p$, the set $Ac$ respects $\overline{a}_{i}$.  
\end{enumerate}
\end{defn}

\begin{rem} \label{finite to infinite}
In Definition \ref{respectdef}(1), by the regularity of $\lambda_{i+1}$, we could have equivalently asked for the existence of such an $\alpha < \lambda_{i+1}$ for any $C \subseteq A$ with $|C| < \lambda_{i+1}$, since there are fewer than $\lambda_{i+1}$ finite subsets of any such $C$. 
\end{rem}

\begin{defn} \label{K def}
We define $\mathbb{K}$ to be the class of $\overline{A}$ such that:
\begin{enumerate}
\item $\overline{A} = \langle A_{i} : i \leq \kappa \rangle$ is increasing continuous.
\item $|A_{i}| = \lambda_{i}$ for all $i < \kappa$.
\item $\overline{a}_{i} \subseteq A_{i+1}$ for all $i < \kappa$.
\item $A_{i}$ respects $\overline{a}_{i}$ for all $i < \kappa$, i.e. there is some $\alpha < \lambda_{i+1}$ such that $\overline{a}_{i, \geq \alpha}$ is $A_{i}$-indiscernible, using Remark \ref{finite to infinite}.    
\end{enumerate}
Given $\overline{A},\overline{B} \in \mathbb{K}$, we say $\overline{A} \leq_{\mathbb{K}} \overline{B}$ if $A_{j} \subseteq B_{j}$ for all $j < \kappa$.  We say $\overline{A} \leq_{\mathbb{K},i} \overline{B}$ if $A_{j} \subseteq B_{j}$ for all $j$ satisfying $i \leq j < \kappa$ and $\overline{A} \leq_{\mathbb{K},*} \overline{B}$ if $\overline{A} \leq_{\mathbb{K},i} \overline{B}$ for some $i < \kappa$.  We may omit the $\mathbb{K}$ subscript when it is clear from context.  
\end{defn}

\begin{lem} \label{density replacement}
Suppose $p$ is a partial 1-type over $A_{\kappa}$ with $|\text{dom}(p)| \leq \lambda_{i}$ for some $i < \kappa$.  Then there are $i'$ with $i \leq i' < \kappa$ and $p' \supseteq p$ with $|\text{dom}(p')| \leq \lambda_{i'}$ such that, if $q$ is a type over $A_{\kappa}$ extending $p'$, then $q$ does not shred over $A_{i'}$.  
\end{lem}

\begin{proof}
Suppose not.  Then we will construct an increasing sequence of types $\langle p_{j} : j < \kappa \rangle$ extending $p$ and an increasing sequence of ordinals $\langle i_{j} : j < \kappa \rangle$ such that $|\mathrm{dom}(p_{j})| = \lambda_{i_{j}}$ and $p_{j}$ shreds over $A_{i_{j}}$ for all $j < \kappa$.  To begin, we set $i_{0} = i$ and use our assumption to find some $p_{0} \supseteq p$ such that $p_{0}$ shreds over $A_{i_{0}}$.  We may assume $\text{dom}(p_{0})$ contains $A_{i_{0}}$ and has cardinality $\lambda_{i_{0}}$.  Given any $\langle p_{j} : j < \alpha \rangle$ and $\langle i_{j} : j < \alpha \rangle$ for $\alpha \geq 1$, we put $p' = \bigcup_{j < \alpha} p_{j}$ and $i' = \sup_{j < \alpha} i_{j}$ (here we make use of the fact that $\kappa$ is regular).  Then $|\mathrm{dom}(p')| = \lambda_{i'}$ and $p'$ extends $p$.  Let $i_{\alpha} = i'+1$.  As $i_{\alpha} \geq i$, by hypothesis, there is some type $p_{\alpha} \supseteq p'$ such that $p_{\alpha}$ shreds over $A_{i'+1}$.  As this will be witnessed by a single formula, we may assume $\mathrm{dom}(p_{\alpha})$ contains $A_{i_{\alpha}}$ and $|\mathrm{dom}(p_{\alpha})| = \lambda_{i_{\alpha}}$, completing the induction.

Let $p_{*} = \bigcup_{j < \kappa} p_{j}$.  Then, by construction, we have $p_{*}$ shreds over $A_{i_{j}}$ for all $j < \kappa$.  By Proposition \ref{the main equivalence prop}, this contradicts $\kappa^{1}_{\mathrm{shred}}(T) \leq \kappa$.  
\end{proof}

\begin{lem} \label{the main lemma really}
If $\overline{A} \in \mathbb{K}$ and $p$ is a $1$-type over $A_{\kappa}$ with $|\mathrm{dom}(p)|< \mu$, then there is $\overline{A}' \in \mathbb{K}$ such that $\overline{A} \leq _{\mathbb{K}} \overline{A}'$ and some $c \in A'_{\kappa}$ realizes $p(x)$.  
\end{lem}

\begin{proof}
By Lemma \ref{density replacement} and the choice of $\mu$, we may extend $p$ to a type $p'$ such that, for some $i < \kappa$, $|\mathrm{dom}(p')| \leq \lambda_{i}$  and no type extending $p'$ over $A_{\kappa}$ shreds over $A_{i}$, and hence does not shred over $A_{i'}$ for any $i' \geq i$ by base monotonicity.  Without loss of generality, we may assume $p = p'$.  

By induction on $j \in [i,\kappa]$, we will define types $p_{j} \in S^{1}(A_{j})$ so that 
\begin{enumerate}
\item The types $p_{j}$ are increasing with $j$.  
\item For all $j \in [i,\kappa)$, $p_{j} \cup p$ is consistent.  
\item For all $j \in [i,\kappa)$, if $c \models p_{j+1}$, then for some $\alpha < \lambda_{j+1}$, $\overline{a}_{j, \geq \alpha}$ is $A_{j}c$-indiscernible.  
\end{enumerate}
Let $p_{i} \in S^{1}(A_{i})$ be any type consistent with $p$.  Given $p_{j}$, we note that $p \cup p_{j}$ extends $p$ and therefore does not explicitly shred over $A_{j}$.  Because $|p \cup p_{j}| < \lambda_{j+1}$, by compactness and the fact that $A_{j}$ respects $\overline{a}_{j}$, there is a realization $c \models p \cup p_{j}$ and $\alpha < \lambda_{j+1}$ such that $\overline{a}_{j, \geq \alpha}$ is $A_{j}c$-indiscernible.  We put $p_{j+1} = \text{tp}(c/A_{j+1})$.    Finally, given $\langle p_{j} : j \in [i, \delta) \rangle$ for $\delta$ limit $> i$, we set $p_{\delta} = \bigcup_{j \in [i, \delta)} p_{j}$.  

Define $p_{\kappa} = \bigcup_{j \in [i,\kappa)} p_{j}$.  Let $c$ realize $p_{\kappa}$ and define $\overline{A}_{*}$ by $A^{*}_{j} = A_{j}$ for all $j < i+1$ and $A^{*}_{j} = A_{j} c$ for all $j\geq i+1$.  For all $j \in [i,\kappa)$, as $c$ realizes $p_{j+1}$, we know there is $\alpha < \lambda_{j+1}$ such that $\overline{a}_{j, \geq \alpha}$ is $cA_{j}$-indiscernible.  It follows that $\overline{A}_{*} \in \mathbb{K}$, completing the proof.  
\end{proof}

%
%
%
%

\subsection{A one variable theorem}

\begin{thm} \label{one var thm}
For all $m$, we have $\kappa^{m}_{\mathrm{shred}}(T) = \kappa^{1}_{\mathrm{shred}}(T)$.  
\end{thm}

\begin{proof}
The inequality $\kappa^{m}_{\mathrm{shred}}(T) \geq \kappa^{1}_{\mathrm{shred}}(T)$ is clear, so it suffices to show $\kappa^{1}_{\mathrm{shred}}(T)  \geq \kappa^{m}_{\mathrm{shred}}(T) $.  Suppose $\kappa \geq \kappa^{1}_{\mathrm{shred}}(T)$ is a regular cardinal, $\langle \lambda_{i} : i < \kappa \rangle$ is an increasing continuous sequence of cardinals with $\lambda_{0} > \kappa + |T|$ and $\lambda_{i+1}$ regular for all $i < \kappa$.  Let $\mu = \sup_{i < \kappa} \lambda_{i}$.  Note $\mu > \kappa$.  

We will prove by induction on $m$ that, if $\kappa < \kappa^{m}_{\mathrm{shred}}(T)$, there is an increasing and continuous sequence of sets $\langle B_{i} : i \leq \kappa \rangle$ and $q(y) \in S_{1}(B_{\kappa})$ such that $q \upharpoonright B_{i+1}$ shreds over $B_{i}$.  This contradictions our assumption that $\kappa \geq \kappa^{1}_{\mathrm{shred}}(T)$, by Proposition \ref{the main equivalence prop}.  

When $m = 1$, we immediately have a contradiction since $\kappa^{1}_{\mathrm{shred}}(T) \leq \kappa < \kappa^{1}_{\mathrm{shred}}(T)$. 

Suppose it has been proven for $m$ and suppose $\langle A_{i} : i \leq \kappa \rangle$ is an increasing continuous sequence of models with $|A_{i}| = \lambda_{i}$ and $p(x_{0},\ldots, x_{m}) \in S^{m+1}(A_{\kappa})$ is a type such that $p \upharpoonright A_{i+1}$ shreds over $A_{i}$ with a built-in witness $\overline{b}_{i}$, witnessed by the formula $\varphi_{i}(x_{0},\ldots, x_{m};a_{i}) \in p\upharpoonright A_{i+1}$.  Then because $\overline{b}_{i}$ is $A_{i}$-indiscernible, we have $\langle A_{i} : i \leq \kappa \rangle \in \mathbb{K}$ in the notation of Subsection \ref{respect subsection} with the $\overline{b}_{i}$ playing the role of $\overline{a}_{i}$.  

Let $p'(x_{0},\ldots,x_{m}) = \{\varphi(x_{0},\ldots, x_{m};a_{i}) : i < \kappa\}$ and let $p''(x_{m})$ be defined by 
\begin{eqnarray*}
p''(x_{m}) &=& (\exists x_{0},\ldots, x_{m-1})\bigwedge p'(x_{0},\ldots, x_{m}) \\
&=& \{ (\exists x_{0},\ldots, x_{m-1})\bigwedge_{\varphi \in w} \varphi(x_{0},\ldots, x_{m-1}) : w \subseteq p' \text{ finite}\}.     
\end{eqnarray*}
Note that $|p''| = \kappa < \mu$.  By Lemma \ref{the main lemma really}, there is $\overline{B} = \langle B_{i} : i \leq \kappa \rangle \in \mathbb{K}$ such $\overline{A} \leq_{\mathbb{K}} \overline{B}$ and such that $p''$ is realized by some $c \in B_{\kappa}$.  By the definition of $\mathbb{K}$, for each $i < \kappa$, there is some $\alpha_{i} < \lambda_{i+1}$ such that $\overline{b}_{i, \geq \alpha_{i}}$ is $B_{i}$-indiscernible.  Let $i_{*}$ be minimal such that $c \in B_{i_{*}}$ and let $q(x_{0},\ldots, x_{m-1}) = p'(x_{0},\ldots, x_{m-1},c)$.  Let $q' \in S(B_{\kappa})$ be any completion of $q$. Then for all $i \geq i_{*}$, we have that $q' \upharpoonright B_{i+1}$ shreds over $B_{i}$ with the built-in witness $\overline{b}_{i, \geq \alpha_{i}}$.  Reindexing by setting $B'_{i} = B_{i_{*}+i}$ and $a_{i,j} = b_{i,\alpha_{i} + j}$ for all $i < \kappa$ and $j < \lambda_{i+1}$, we may apply the induction hypothesis to complete the proof.  
\end{proof}

\subsection{Exact saturation}

As in Subsection \ref{respect subsection}, we fix a singular cardinal $\mu$.  Writing $\text{cf}(\mu) = \kappa$, we will assume there is an increasing and continuous sequence of cardinals $\overline{\lambda} = \langle \lambda_{i} : i \leq \kappa \rangle$ such that $\lambda_{0} > \kappa$, $\lambda_{i+1}$ is regular for all $i < \kappa$, and $\lambda_{\kappa} = \mu$.  

We write $I$ to denote $\{(i,\alpha) : i < \kappa, \alpha < \lambda_{i+1}\}$ ordered lexicographically.  We write $I_{i, \geq\beta} = \{(j,\alpha) : j = i \text{ and }\alpha \geq \beta\}$ and we write $I_{i}$ for $I_{i,\geq 0}$.  We also fix an indiscernible sequence $\overline{a} = \langle a_{t} : t \in I\rangle$.  We similarly write $\overline{a}_{i,\geq \beta}$ for $\langle a_{t} : t \in I_{i, \geq\beta} \rangle$ and $\overline{a}_{i}$ for $\langle a_{t} : t \in I_{i} \rangle$.  If $i < \kappa$, and $\alpha < \beta < \lambda_{i+1}$, we write $\overline{a}_{i,\alpha,\beta}$ for the sequence $\langle a_{j,\gamma} : j = i, \gamma \in [\alpha, \beta)\rangle$.  Note that, in particular, we have $\overline{a}_{i}$ is $\overline{a}_{<i}$-indiscernible.  In this subsection, we will write $\mathbb{K}$ to refer to the class of $\overline{A}$ as in Definition \ref{K def} with respect to the sequences $\overline{a}_{i}$ described above.  

Additionally, we will assume that $T$ is a theory with $\kappa^{1}_{\mathrm{shred}}(T) \leq \kappa = \text{cf}(\mu)$ and with the independence property witnessed by the formula $\varphi(x;y)$ along the sequence $\langle a_{i} : i \in I \rangle$\textemdash that is, for all $X \subseteq I$, we have that $\{\varphi(x;a_{i})^{(\mathrm{if }i \in X)} : i \in I\}$ is consistent. 

We will construct a model containing $\langle a_{i} : i \in I \rangle$ that is $\mu$-saturated but every finite tuple from this model has the property that there are intervals from our fixed indiscernible sequence $\langle a_{i} : i \in I \rangle$ which are indiscernible over it.  Because we assume $T$ has the independence property, witnessed along this indiscernible sequence, it will follow that $\{\varphi(x;a_{i}) : i \text{ even}\} \cup \{\neg \varphi(x;a_{i}) : i \text{ odd}\}$ is an omitted type, which means that the model produced by our construction is not $\mu^{+}$-saturated.  Our proof pursues the same strategy as the construction of an exactly satured model of a simple theory from \cite[Theorem 3.3]{20-Kaplan2015}, but with $\kappa_{\mathrm{shred}}(T) < \infty$ replacing the assumption of simplicity.  

In order to organize the construction, we will use the following combinatorial principle:  

\begin{defn}
Suppose $\kappa$ is an uncountable cardinal.  For a club $C$, we write $\mathrm{Lim}(C)$ for the set $\{\alpha  \in C : \sup(C \cap \alpha) = \alpha\}$.  We write $\square_{\kappa}$ for the following assertion:  there is a sequence $\langle C_{\alpha} : \alpha \in \text{Lim}(\kappa^{+})\rangle$ such that
\begin{enumerate}
\item $C_{\alpha} \subseteq \alpha$ is club.
\item If $\beta \in \text{Lim}(C_{\alpha})$ then $C_{\beta} = C_{\alpha} \cap \beta$.  
\item If $\text{cf}(\alpha) < \kappa$, then $|C_{\alpha}| < \kappa$.  
\end{enumerate}
We call such a sequence a $\emph{square sequence}$ (for $\kappa$).  
\end{defn}

The following remark was noted in \cite[Remark 3.2]{20-Kaplan2015} \textemdash it will play a similar role in our deduction of the main theorem.

\begin{rem}\label{bettersquare}
Suppose $\langle C_{\alpha} : \alpha \in \text{Lim}(\kappa^{+}) \rangle$ is a square sequence and $C'_{\alpha} = \text{Lim}(C_{\alpha})$.  Then we have the following:
\begin{enumerate}
\item If $C'_{\alpha} \neq \emptyset$ and $\sup(C'_{\alpha}) \neq \alpha$ then $C'_{\alpha}$ has a last element and $\text{cf}(\alpha) = \omega$.  If $C'_{\alpha} = \emptyset$ then $\text{cf}(\alpha) = \omega$.
\item  For all $\beta \in C'_{\alpha}$, $C'_{\beta} = C_{\alpha}' \cap \beta$.  
\item If $\text{cf}(\alpha) < \kappa$, then $|C'_{\alpha}| < \kappa$.  
\end{enumerate}
\end{rem}

The following is the main theorem of the section.  The proof follows \cite[Theorem 3.3]{20-Kaplan2015}. 

\begin{thm}
If $T$ has the independence property and $\kappa_{\mathrm{shred}}(T) < \infty$, then $T$ has an exactly $\mu$-saturated model for any singular $\mu > |T|$ of cofinality $\kappa \geq \kappa_{\mathrm{shred}}(T)$ such that $\square_{\mu}$ and $2^{\mu} = \mu^{+}$.  
\end{thm}

\begin{proof}
Let $\langle C_{\alpha} : \alpha \in \text{Lim}(\mu^{+}) \rangle$ be a sequence as in Remark \ref{bettersquare}.  Note that, for all $\alpha \in \mathrm{Lim}(\mu^{+})$, we have that $|C_{\alpha}| < \mu$ by condition (3) of Remark \ref{bettersquare}, as $\alpha < \mu^{+}$ and hence $\mathrm{cf}(\alpha) < \mu$, as $\mu$ is singular.  Partition $\mu^{+}$ into $\{S_{\alpha} : \alpha < \mu^{+}\}$ so that each $S_{\alpha}$ has size $\mu^{+}$.  By induction, we will construct a sequence of pairs $\langle (\overline{A}_{\alpha}, p_{\alpha}) : \alpha < \mu^{+} \rangle$ such that 
\begin{enumerate}
\item $\overline{A}_{\alpha} = \langle A_{\alpha,i} : i < \kappa \rangle \in \mathbb{K}$.
\item $\overline{p}_{\alpha} = \langle p_{\alpha,\beta} : \beta \in S_{\alpha} \setminus \alpha \rangle$ is an enumeration of all complete 1-types over subsets of $\bigcup_{i} A_{\alpha,i}$ of size $< \mu$ (using $|T| < \mu$ and $2^{\mu} = \mu^{+})$.  
\item If $\beta < \alpha$, then $\overline{A}_{\beta} \leq_{*} \overline{A}_{\alpha}$.  
\item If $\alpha \in S_{\gamma}$ and $\gamma < \alpha$, then $\overline{A}_{\alpha+1}$ contains a realization of $p_{\gamma,\alpha}$.  
\item If $\alpha$ is a limit, then for any $i < \kappa$ such that $|C_{\alpha}| < \lambda_{i}$ and $\beta \in C_{\alpha}$, then we have that $\overline{A}_{\beta} \leq_{i} \overline{A}_{\alpha}$.  
\end{enumerate}
At stage $0$, we define $\overline{A}_{0}$ to be the minimal sequence in $\mathbb{K}$\textemdash that is, $A_{0,i} = \bigcup \overline{a}_{<i}$ for all $i < \kappa$.  For the successor case, use Lemma \ref{the main lemma really}.  

Now we handle the limit cases.  

\textbf{Case 1}:  $\sup(C_{\alpha}) = \alpha$.  Let $i_{0} = \text{min}\{i < \kappa : |C_{\alpha}| < \lambda_{i}\}$ which is necessarily a successor ordinal.  For $i < i_{0}$, we define $A_{\alpha,i} = \overline{a}_{<i}$ and for $i \geq i_{0}$ successor, we let $A_{\alpha,i} = \bigcup_{\beta \in C_{\alpha}} A_{\beta,i}$.  Note that $|A_{\beta,i}| \leq \lambda_{i}$ for all $i < \kappa$, and for $i$ limit we define $A_{\alpha,i}$ by continuity, setting
$$
A_{\alpha,i} = \bigcup_{\substack{j < i \\ j \text{ successor}}} A_{\alpha,j}.
$$
Note that it follows, then, that for $i$ limit, we also have $A_{\alpha,i} = \bigcup_{\beta \in C_{\alpha}} A_{\beta,i}$.  

We have to check (1),(3), and (5).  First we show that $\overline{A}_{\alpha} \in \mathbb{K}$.  The only thing to check is that $i \geq i_{0}$ implies $A_{\alpha,i}$ respects $\overline{a}_{i}$.  Now if $w \subseteq A_{\alpha,i}$ is a finite set, for each $e \in w$, there is some $\beta_{e} \in C_{\alpha}$ so that $e \in A_{\beta_{e},i}$.  Let $\beta = \max\{\beta_{e} : e \in w\}$.  Then $C_{\alpha} \cap \beta = C_{\beta}$.  By (5), the fact that $|C_{\beta}| < \lambda_{i_{0}}$, and induction, we have $\beta_{e} < \beta$ implies $\beta_{e} \in C_{\beta}$ and $\overline{A}_{\beta_{e}} \leq_{i_{0}} \overline{A}_{\beta}$ so $A_{\beta_{e},i} \subseteq A_{\beta,i}$.  It follows that $w \subseteq A_{\beta,i}$.  As $\overline{A}_{\beta} \in \mathbb{K}$, we know $A_{\beta,i}$ respects $\overline{a}_{i}$, so there is some $\delta < \lambda_{i+1}$ such that $\overline{a}_{i,\geq \delta}$ is $w$-indiscernible.  As $w \subseteq A_{\alpha,i}$ is arbitrary, this shows $A_{\alpha,i}$ respects $\overline{a}_{i}$ and, therefore, $\overline{A}_{\alpha} \in \mathbb{K}$.  Next, if $\beta < \alpha$, then, because $\text{sup}(C_{\alpha}) = \alpha$, there is $\beta' \in C_{\alpha}$ such that $\beta < \beta'$.  By induction, $\overline{A}_{\beta} \leq_{*} \overline{A}_{\beta'}$ and, by construction, $\overline{A}_{\beta'} \leq_{i_{0}} \overline{A}_{\alpha}$, from which it follows that $\overline{A}_{\beta} \leq_{*} \overline{A}_{\alpha}$, which shows (3).  Finally (5) is by construction.

\textbf{Case 2}:   $\text{sup}(C_{\alpha}) < \alpha$.  We know in this case $C_{\alpha}$ has a maximum element $\gamma$ and $\text{cf}(\alpha) = \omega$.  Choose an increasing cofinal sequence $\langle \beta_{n} : n < \omega \rangle$ in $\alpha$ with $\beta_{0} = \gamma$.  Then, by induction, we may choose an increasing sequence of successor ordinals $\langle i_{n} : n < \omega \rangle$ so that $\overline{A}_{\beta_{n}} \leq_{i_{n}} \overline{A}_{\beta_{n+1}}$.  Setting $i_{-1} = 0$ and $i = \text{sup}\{i_{n}: n < \omega\}$, we define $\overline{A}_{\alpha}$ as follows:  for successor $j \in [i_{n-1},i_{n})$, we put $A_{\alpha,j} = A_{\beta_{n},j}$ and for successor $j \geq i$, we put $A_{\alpha,j} = \bigcup_{n < \omega} A_{\beta_{n},j}$.  For limit ordinals $j$, $A_{\alpha,j}$ is defined by continuity.  It is easy to see that this satisfies (1) and (3), so we check (5).

First, observe that $\overline{A}_{\gamma} \leq \overline{A}_{\alpha}$.  To see this, it suffices to show by induction on $n$, that if $j \geq i_{n-1}$, then $A_{\gamma,j} \subseteq A_{\beta_{n},j}$.  For $n=0$ this is by definition.  Assuming it is true for $n$, we can consider an arbitrary $j > i_{n}$.  Then by choice of $i_{n}$, $\overline{A}_{\beta_{n}} \leq_{i_{n}} \overline{A}_{\beta_{n+1}}$ so $A_{\beta_{n},j} \subseteq A_{\beta_{n+1},j}$.  As the sequence $\langle i_{n} : n < \omega \rangle$ is increasing, we have also $j > i_{n-1}$ so, by the inductive hypothesis, $A_{\gamma,j} \subseteq A_{\beta_{n},j}$ so, by transitivity, $A_{\gamma,j} \subseteq A_{\beta_{n+1},j}$ as desired.

Now suppose $i < \kappa$, $|C_{\alpha} | < \lambda_{i}$, and $\beta \in C_{\alpha}$.  Then $\beta \leq \gamma$ and as $\overline{A}_{\gamma} \leq \overline{A}_{\alpha}$ we have in particular that $\overline{A}_{\gamma} \leq_{i} \overline{A}_{\alpha}$, so we may assume $\beta < \gamma$.  Then $\beta \in C_{\alpha} \cap \gamma = C_{\gamma}$ and $|C_{\gamma}| = |C_{\alpha} \cap \gamma| < \lambda_{i}$ so it follows by induction that $\overline{A}_{\beta} \leq_{i} \overline{A}_{\gamma} \leq \overline{A}_{\alpha}$ so $\overline{A}_{\beta} \leq_{i} \overline{A}_{\alpha}$.  

To conclude, we define a model $M$ by 
$$
M = \bigcup_{\substack{\alpha < \mu^{+} \\ i < \kappa}} A_{\alpha,i}.
$$
By (4), the model $M$ is $\mu$-saturated.  Moreover $M$ is not $\mu^{+}$-saturated, as the partial type 
$$
\{\varphi(x;a_{i,\alpha}) : i < \kappa, \alpha \text{ even}\} \cup \{\neg \varphi(x;a_{i,\alpha}) : i < \kappa, \alpha \text{ odd}\}
$$
is omitted by (1). 
\end{proof}

\begin{quest}
Suppose $T$ is NTP$_{2}$ and has the independence property, and assume $\mu$ is a singular cardinal such that $\mathrm{cf}(\mu) > |T|$, $2^{\mu} = \mu^{+}$, and $\square_{\mu}$.  Does $T$ have an exactly $\mu$-saturated model?
\end{quest}

\section{Examples}

\subsection{Standard examples for the SOP$_{n}$ hierarchy}

Recall the definition of the SOP$_{n}$ heirarchy:

\begin{defn}
Suppose $n \geq 3$.  The theory $T$ has the $n$\emph{th strong order property} (SOP$_{n})$ if there is a formula $\varphi(x;y)$ and a sequence of tuples $\langle a_{i} : i < \omega \rangle$ so that 
\begin{enumerate}
\item $\models \varphi(a_{i};a_{j})$ if and only if $i < j$.
\item $\{\varphi(x_{i},x_{i+1}) : i < n-1\} \cup \{\varphi(x_{n-1},x_{0})\}$ is inconsistent.  
\end{enumerate}  
If $T$ does not have SOP$_{n}$, we say $T$ is NSOP$_{n}$.  
\end{defn}

Note that $\mathrm{SOP}_{n+1} \implies \mathrm{SOP}_{n}$ for all $n \geq 3$ \cite[Claim 2.6]{shelah1996toward}.

By a directed graph we mean a set with a binary relation that is assymetric and irreflexive.  Given a natural number $n \geq 3$, we let $L_{n} = \{R_{1}(x,y)\} \cup \{S_{l}(x,y) : 1 \leq l < n\}$ be a language with $n$ binary relations.  The theory $T^{0}_{n}$ is the $L_{n}$-theory of directed graphs with no cycle of length $\leq n$, where $R_{1}(x,y)$ is the (assymetric) edge relation and $S_{l}(x,y)$ means that there is no directed path in the graph $R_{1}$ of length $\leq l$ from $x$ to $y$.  More precisely, $T^{0}_{n}$ consists of the following axioms:
\begin{itemize}
\item $R_{1}(x,y)$ is an irreflexive assymetric relation:
$$
(\forall x,y)[R_{1}(x,y) \to \neg R_{1}(y,x)].
$$ 
\item There are no directed loops of length $\leq n$.  That is, for all $k$ with $1 \leq k \leq n$, we have
$$
\neg (\exists z_{0},\ldots, z_{k-1}) \left[\bigwedge_{i < k-1} R_{1}(z_{i},z_{i+1}) \wedge R_{1}(z_{k-1},z_{0}) \right].
$$
\item The relation $S_{l}(x,y)$ implies that there is no directed path of positive length $\leq l$ from $x$ to $y$:
$$
(\forall x,y) \left[ S_{l}(x,y) \to \neg (\exists z_{0},\ldots, z_{l}) \left[ z_{0} = x \wedge z_{l} = y \wedge \bigwedge_{i < l} R_{1}(z_{i},z_{i+1}) \vee z_{i} = z_{i+1} \right] \right].
$$
\item Paths satisfy the triangle inequality:  if $l+l' < n$, then 
$$
(\forall x,y,z)\left[\neg S_{l}(x,y) \wedge \neg S_{l'}(y,z) \to \neg S_{l+l'}(x,z) \right],
$$
and, because there are no loops of size $\leq n$, for all $1 \leq l < n' \leq n$
$$
(\forall x,y,z)\left[\neg S_{l}(x,y) \to S_{n'-l}(y,x)\right].
$$
\end{itemize}  
This is a universal theory and the model completion of $T^{0}_{n}$ is denoted $T_{n}$\textemdash it eliminates quantifiers.  Note that $R_{1}(x,y)$ is equivalent to $\neg S_{1}(x,y)$.  We will write $R_{l}(x,y)$ for $\neg S_{l}(x,y)$, which indicates there is a directed path of length $\leq l$ from $x$ to $y$.  We will write $\mathbb{M}_{n} \models T_{n}$ for the monster model of $T_{n}$.  The existence of the model completion is proved in \cite[Claim 2.8(3)]{shelah1996toward}, where it is also shown that $T_{n}$ is SOP$_{n}$ and NSOP$_{n+1}$.  

\begin{prop}
If $n \geq 4$, then $\kappa_{\mathrm{shred}}(T_{n}) = \infty$.
\end{prop}

\begin{proof}
Let $\kappa$ be an arbitrary infinite regular cardinal.  Define a directed graph $G$ with domain $\{b_{i,\alpha} : i< \kappa, \alpha < \omega\} \cup \{a_{i,j}: i < \kappa, j < 2\}$ and interpret the edge relation $R_{1}$ in $G$ by 
$$
R_{1}^{G} = \{(a_{i,0},b_{i,\alpha}) : i < \kappa, \alpha < \omega \text{ even}\} \cup \{(b_{i,\alpha},\alpha_{i,1}) : i < \kappa, \alpha < \omega \text{ odd}\},
$$
and then interpret $S^{G}_{l}$ and hence $R^{G}_{l}$ for $1 \leq l < n$ according to the axioms.  This clearly defines a model of $T^{0}_{n}$ so there is an $L_{n}$-embedding of $G$ into the monster model $\mathbb{M}_{n} \models T_{n}$.  Therefore, we may identify $G$ with an $L_{n}$-substructure of $\mathbb{M}$.   Define $A_{i} = \overline{a}_{\leq i} \overline{b}_{\leq i}$, for all $i < \kappa$.  

Let $\varphi(x;y,z) = R_{1}(x,y) \wedge R_{1}(z,x)$ and define a partial type $p$ by $p  = \{\varphi(x;a_{i,0},a_{i,1}) : i < \kappa\}$.  It is clear from the construction of $G$ that any vertex satisfying this collection of formulas would not create a cycle, hence in particular, it will not create a cycle of length $\leq n$ and, therefore, $p$ is a consistent set of formulas.  


Fix $i < \kappa$.  By quantifier-elimination, we have $\overline{b}_{i+1} = (b_{i+1,\alpha})_{\alpha < \omega}$ is $A_{i}$-indiscernible.  Let $c$ realize $\varphi(x;a_{i+1,0},a_{i+1, 1})$.  Then we have 
\begin{enumerate}
\item $R_{2}(c,b_{i+1,\alpha})$ for $\alpha < \omega$ even.
\item $R_{2}(b_{i+1,\alpha},c)$ for $\alpha < \omega$ odd.
\item $\{R_{2}(x,b_{i+1, \alpha}), R_{2}(b_{i+1,\alpha},x)\}$ is inconsistent for all $\alpha$, because $n \geq 4$.   
\end{enumerate}
It follows that no end-segment of $\overline{b}_{i+1}$ can be $c$-indiscernible, and therefore cannot be $A_{i}c$-indiscernible.  In fact, $\varphi(x;a_{i+1,0},a_{i+1,1}) \vdash R_{2}(x;b_{i+1,\alpha}) \leftrightarrow \neg R_{2}(x;b_{i+1,\alpha+1})$ for all even $\alpha < \omega$, which shows that $\varphi(x;a_{i+1,0},a_{i+1,1})$ explicitly shreds over $A_{i}$.  It follows that $\kappa_{\mathrm{shred}}(T) > \kappa$ and, as $\kappa$ was arbitrary, we have $\kappa_{\mathrm{shred}}(T) = \infty$.  
\end{proof}

Now we analyze $T_{3}$:  

\begin{lem} \label{T3 algebraic}
In $T_{3}$, if $\overline{b} = \langle b_{i} : i < \lambda \rangle$ is indiscernible over $A$, then for any tuple $a$, if $c$ is a tuple disjoint from $Aa$, then there is $c' \equiv_{Aa} c$ so that $\overline{b}$ is $Ac'$-indiscernible.
\end{lem}

\begin{proof}
Note that $T_{3}$ eliminates quantifiers in the language containing only $R_{1}$, since $R_{2}(x,y)$ is definable by the formula $x \neq y \wedge \neg R_{1}(y,x)$.  For simplicity, we will write $R$ for $R_{1}$.  Because algebraic closure in $T_{3}$ is trivial, by replacing $c$ by something with the same type over $Aa$, we may assume $c$ is disjoint from $Aa\overline{b}$.  Define a model $M \models T^{0}_{3}$ as follows with underlying set $Aa\overline{b}c$ by defining
$$
R^{M} = R^{\mathbb{M}_{3}} \upharpoonright Aa\overline{b} \cup R^{\mathbb{M}_{3}} \upharpoonright Aac.  
$$	
We claim that $M \models T^{0}_{3}$.  To see this, suppose not and there are distinct $d_{0},d_{1},d_{2} \in M$ so that $R^{M}(d_{0},d_{1})$, $R^{M}(d_{1},d_{2})$, and $R^{M}(d_{2},d_{0})$.  Since $\mathbb{M}_{3}$ has no directed cycles of length 3, it is impossible for $d_{0},d_{1},d_{2}$ to be all contained in $Aa\overline{b}$ or all contained in $Aac$.  Therefore, without loss of generality, $d_{0} \in Aa\overline{b} \setminus Aac$.  But then since $R^{M}(d_{2},d_{0})$ and $R^{M}(d_{0},d_{1})$, we have $d_{1},d_{2} \in Aa\overline{b}$, by the definition of $R^{M}$, a contradiction.  This shows $M$ has no directed cycle of length 3 so $M \models T^{0}_{3}$.  

Embed $M$ into $\mathbb{M}_{3}$ over $Aa\overline{b}$ and let $c'$ be the image of $c$.  By quantifier elimination, we have $c' \equiv_{Aa} c$ and, because $c'$ is disjoint from $Aa\overline{b}$, we have $\overline{b}$ is $Ac'$-indiscernible.  
\end{proof}

\begin{prop}
$\kappa_{\mathrm{shred}}(T_{3}) =  \aleph_{0}$.	
\end{prop}

\begin{proof}
%
By Theorem \ref{one var thm}, it suffices to show $\kappa^{1}_{\mathrm{shred}}(T_{3}) \leq \aleph_{0}$, and, in fact, we will show there is no shredding chain in a single free variable of length $2$.  Towards contradiction, suppose $A$ is a set of parameters, $\varphi_{0}(x;a_{0})$ shreds over $A$ witnessed by $\overline{b}_{0}$, $\varphi_{1}(x;a_{1})$ shreds over $Aa_{0}$ witnessed by $\overline{b}_{1}$, and $\{\varphi_{0}(x;a_{0}), \varphi_{1}(x;a_{1})\}$ is consistent, with $x$ a single free variable.  Because $\varphi_{0}(x;a_{0})$ has no realization $c$ such that $\overline{b}_{0}$ is indiscernible over $Ac$, it follows by Lemma \ref{T3 algebraic} that any realization of $\varphi_{0}(x;a_{0})$ is contained in $Aa_{0}$.  Then let $c \models \{\varphi_{0}(x;a_{0}),\varphi_{1}(x;a_{1})\}$.  Because $c$ is an element of $Aa_{0}$, it follows that $\overline{b}_{1}$ is $Aa_{0}c$-indiscernible, contradicting the fact that $\overline{b}_{1}$ witnesses that $\varphi_{1}(x;a_{1})$ shreds over $Aa_{0}$.  This completes the proof.    
\end{proof}

\subsection{NSOP$_{1}$ and unshreddability}

There is a theory of independence for NSOP$_{1}$ theories that indicates this class of theories may be considered quite close to the class of simple theories (see, e.g., \cite{kaplan2017kim}).  In the next two examples, however, we show that unshreddability is independent of NSOP$_{1}$ and, in particular, that within the class of NSOP$_{1}$ theories, it is still possible that $\kappa_{\mathrm{shred}}(T) = \infty$.  Recall the definition of SOP$_{1}$:
\begin{defn}
A formula $\varphi(x;y)$ is said to have SOP$_{1}$ if there is a tree of tuples $(a_{\eta})_{\eta \in 2^{<\omega}}$ satisfying the following:
\begin{enumerate}
\item For all $\eta \in 2^{\omega}$, $\{\varphi(x;a_{\eta | \alpha}) : \alpha < \omega\}$ is consistent.
\item For all $\eta \perp \nu$ in $2^{<\omega}$, if $(\eta \wedge \nu) \frown 0 \unlhd \eta$ and $(\eta \wedge \nu) \frown 1 = \nu$, then $\{\varphi(x;a_{\eta}), \varphi(x;a_{\nu})\}$ is inconsistent.  
\end{enumerate}
A theory $T$ is said to have SOP$_{1}$ if some $\varphi(x;y)$ has SOP$_{1}$ modulo $T$, otherwise $T$ is NSOP$_{1}$.  
\end{defn}

First, we describe an NSOP$_{1}$ example of a theory $T^{*}_{1}$ with $\kappa_{\mathrm{shred}}(T^{*}_{1}) = \aleph_{0}$.  This theory was studied in detail in \cite[Subsection 9.2]{kaplan2017kim}.  The language $L_{1}$ consists of unary predicates $F$ and $O$, a binary relation $E$, and a binary function $\text{eval}$.  The theory $T_{1}$ consists of the following axioms:
\begin{enumerate}
\item $F$ and $O$ partition the universe.
\item $E \subseteq O^{2}$ is an equivalence relation.
\item $\text{eval}: F \times O \to O$ is a selector function:
\begin{enumerate}
\item $(\forall x \in F)(\forall y \in O)\left[E(y,\text{eval}(x,y))\right]$.
\item $(\forall x \in F)(\forall y,z \in O)\left[E(y,z) \to \text{eval}(x,y) = \text{eval}(x,z)\right]$.
\end{enumerate}
\end{enumerate}

It was shown in \cite[Subsection 9.2]{kaplan2017kim} that $T_{1}$ has a model-completion $T^{*}_{1}$, which is the theory of the Fra\"iss\'e limit of finite models of $T_{1}$, which is $\aleph_{0}$-categorical with elimination of quantifiers.  It is additionally shown that algebraic closure and definable closure of a set coincide with the structure generated by the set, and that the theory is non-simple NSOP$_{1}$. 

\begin{exmp}
Suppose $c \in O$ and $\langle b_{i} : i < \lambda \rangle$ is an indiscernible sequence such that $b_{i} \in F$ for all $i < \lambda$ and $\mathrm{eval}(b_{i},c) = c$ if $i$ is even and the $\mathrm{eval}(b_{i},c)$ pairwise distinct and different from $c$ for $i$ odd.  Then the formula $E(x;c)$ implies $\mathrm{eval}(b_{i},x) = \mathrm{eval}(b_{i},c)$ for all $i$, thus $E(x,c)$ shreds over $\emptyset$, since for any even $\alpha < \lambda$ and $d$ with $\models E(d,c)$, $\mathrm{eval}(b_{\alpha},d) = \mathrm{eval}(b_{\alpha+2},d)$ and $\mathrm{eval}(b_{\alpha},d) \neq \mathrm{eval}(b_{\alpha+1},d)$.  
\end{exmp}

We show that, in a sense made precise by the following lemma, all instances of shredding in the theory $T^{*}_{1}$ resemble the previous example.  

\begin{lem} \label{shredding char}
Suppose $\varphi(x;a)$ is a non-algebraic formula with $l(x) = 1$ that shreds over $A$.  Then there is some $c \in Aa$, $E$-equivalent to no element of $A$, such that $\varphi(x;a) \vdash E(x,c)$.  In particular $\varphi(x;a)\vdash x \in O$.  
\end{lem}

\begin{proof}
Suppose $\varphi(x;a)$ is a non-algebraic formula with $l(x) = 1$ and $\models \varphi(f;a)$ with $\neg E(f,c)$ for every $c \in Aa$ in an $E$-equivalence class disjoint from $A$.  We must show $\varphi(x;a)$ does not shred over $A$.  As a formula shreds over $A$ only if it shreds over some finite subset of $A$ (i.e. the parameters appearing in the formulas witnessing that it explicitly shreds over $A$), by Lemma \ref{no parameters}, we may assume $A$ is finite and that $a$ enumerates the structure generated by $A$ and $a$.  

Fix an $A$-indiscernible sequence $\overline{b} = (b_{i})_{i < \lambda}$ for $\lambda = (|T| + |A|)^{+}$ where $b_{i} = (b_{i,0},\ldots, b_{i,n-1})$ for all $i < \lambda$.  Additionally, as $a$ is a finite tuple, we can find some ordinal $\alpha$ such that, for each $j < n$, either there is an equivalence class represented by an element of $a$ such that $b_{i,j}$ is in this equivalence class for all $i \geq \alpha$, or $b_{i,j}$ is not equivalent to any element of $a$ for all $i \geq \alpha$; and additionally, either there is an element of $a$ such that $b_{i,j}$ is equal to this element of $a$ $i \geq \alpha$, or $b_{i,j}$ is not equivalent to any element of $a$ for all $j \geq \alpha$ (in other words, we may find some $\alpha$ such that $\overline{b}_{\geq \alpha}$ is $Aa$-indiscernible in the stable reduct $(F,O,E)$, where we forget the function $\mathrm{eval}$).  

Let $B = \langle a \overline{b}_{\geq \alpha} \rangle$ and $C = \langle a f \rangle$ be the structures generated by $a\overline{b}_{\geq \alpha}$ and $af$ in $\mathbb{M}$, respectively.  By the assumption that $\varphi(x;a)$ is not algebraic, we may assume $f \not\in B$. By our assumption that $\neg E(f,c)$ for every $c \in Aa$ in an $E$-equivalence class disjoint from $A$, we have the following three cases: 

\textbf{Case 1}:  $f \in O$ and $f$ is not $E$-equivalent in $\mathbb{M}$ to any element of $\langle a \rangle$.  

In this case, we define a structure $D$ whose underlying set is $B \cup C = B \cup [f]_{E}^{C}$, where $[f]^{C}_{E}$ denotes the $E$-class of $f$ in $C$.  We interpret $F^{D} = F^{B}$ and $O^{D} = O^{B} \cup [f]_{E}^{C}$, then we intepret $E^{D}$ to extend $E^{B}$ with $[f]^{C}_{E}$ forming a new equivalence class (thus also extending $E^{C}$).  Then we define $\mathrm{eval}^{D}$ to extend $\mathrm{eval}^{B}$ and $\mathrm{eval}^{C}$ (which agree on their common domain) and set $\mathrm{eval}(g,f) = f$ for all $g \in F^{B} \setminus F^{C}$.  

\textbf{Case 2}:  $f \in O$ and $f$ is $E$-equivalent to some $c \in O^{\langle A \rangle}$.  

In this case, we define a structure $D$ whose underlying set is $B \cup C  = B \cup \{f\}$.  We interpret $F^{D} = F^{B}$ and $O^{D} = O^{B} \cup \{f\}$, then we intepret $E^{D}$ to extend $E^{B}$ and $E^{C}$ with $[c]^{D}_{E} = [c]^{B}_{E} \cup \{f\}$.  Then we define $\mathrm{eval}^{D}$ to extend $\mathrm{eval}^{B}$ by setting $\mathrm{eval}^{D}(g,f) = \mathrm{eval}^{B}(g,c)$ for all $g \in F^{B}$.  Note that this extends $\mathrm{eval}^{C}$ as well.  

\textbf{Case 3}:  $f \in F$.  

In this case, as before, we define a structure $D$ whose underlying set is $B \cup C \cup \{*_{x} : x \in (O^{B}/E) \setminus (O^{C}/E)\}$ where each $*_{x}$ is a new formal element indexed by an equivalence class of $B$ which is not represented by any element of $\langle a \rangle$.  We interpret $F^{D} = F^{B} \cup F^{C} = F^{B} \cup \{f\}$ and $O^{D} = O^{B} \cup O^{C} \cup  \{*_{x} : x \in (O^{B}/E) \setminus (O^{C}/E)\}$, then we intepret $E^{D}$ to be the equivalence relation generated by $E^{B} \cup E^{C}$ (which extends both $E^{B}$ and $E^{C}$) and the condition that $*_{x}$ is in the equivalence class $x$ for each $x \in (O^{B}/E) \setminus (O^{C}/E)$. Then to define $\mathrm{eval}^{D}$, extending $\mathrm{eval}^{B}$ and $\mathrm{eval}^{C}$, we must define $\mathrm{eval}^{D}(f,c)$ for all $c \in F^{B}$ which are not $E^{D}$-equivalent an element of $C$.  For any such $c$, we define $\mathrm{eval}^{D}(f,c) = *_{[c]_{E}}$.  

In each case, $D$ extends $B$ and $C$ and one can check and $A$-indiscernibility of $\overline{b}$ that, in $D$, $\overline{b}_{\geq \alpha}$ is quantifier-free indiscernible over $C$.  We may embed $D$ into $\mathbb{M}$ over $B$, and then the image $f'$ of $f$ along this embedding satisfies $\models \varphi(f';a)$ and $\overline{b}_{\geq \alpha}$ is $Af'$ indiscernible. This shows that $\varphi(x;a)$ does not shred over $A$.  
\end{proof}

\begin{prop}
The theory $T^{*}_{1}$ is a non-simple NSOP$_{1}$ theory with $\kappa_{\mathrm{shred}}(T^{*}_{1}) = \aleph_{0}$.  
\end{prop}

\begin{proof}
By Theorem \ref{one var thm}, it suffices to show $\kappa^{1}_{\mathrm{shred}}(T^{*}_{1}) = \aleph_{0}$.  Note that if $\varphi_{0}(x;a_{0})$ shreds over $A$ and $\varphi_{1}(x;a_{1})$ shreds over $Aa_{0}$ with $l(x) = 1$ and both $\varphi_{0}$ and $\varphi_{1}$ are non-algebraic, then by Lemma \ref{shredding char}, we must have both that $\varphi_{0}(x;a_{0})$ implies that $x$ is in an equivalence class represented by an element of $a_{0}$ and $\varphi_{1}(x;a_{1})$ implies $x$ is in an equivalence class of an element of $a_{1}$ not represented by an element of $Aa_{0}$.  This implies $\{\varphi_{0}(x;a_{0}), \varphi_{1}(x;a_{1})\}$ is inconsistent.  

Now suppose $\varphi_{i}(x;a_{i})$ are formulas with $l(x) = 1$ for $i = 0,1,2$, such that $\varphi_{i}(x;a_{i})$ shreds over $Aa_{< i}$ for $i = 0,1,2$ and $\{\varphi_{i}(x;a_{i}) : i < 2\}$ is consistent.  Then, by the first paragraph, one of $\varphi_{0}(x;a_{0})$ and $\varphi_{1}(x;a_{1})$ must be algebraic.  Hence if $f \models \{\varphi_{i}(x;a_{i}) : i < 2\}$, then $f \in \mathrm{acl}(Aa_{0}a_{1})$.  But any $Aa_{0}a_{1}$-indiscernible sequence is automatically $\mathrm{acl}(Aa_{0}a_{1})$-indiscernible and therefore $Aa_{0}a_{1}f$-indiscernible.  It follows that $\varphi_{2}(x;a_{2})$ cannot shred over $Aa_{0}a_{1}$, a contradiction.  Therefore $\kappa^{1}_{\mathrm{shred}}(T) = \aleph_{0}$.  
\end{proof}

The following theory is a variation on the generic theory of selector functions $T^{*}_{1}$ considered above.  The language $L$ for our example consists of unary predicates $F,O_{0},O_{1}$, and $O$, binary relations $E, R_{0},$ and $R_{1}$, and a binary function $\text{eval}$.  The theory $T$ consists of the following axioms:
\begin{enumerate}
\item $F$, $O_{0}$, and $O_{1}$ partition the universe and $O = O_{0} \cup O_{1}$.
\item $E \subseteq O^{2}$ is an equivalence relation.
\item $\text{eval}: F \times O \to O_{0}$ is a selector function:
\begin{enumerate}
\item $(\forall x \in F)(\forall y \in O)\left[E(y,\text{eval}(x,y))\right]$.
\item $(\forall x \in F)(\forall y,z \in O)\left[E(y,z) \to \text{eval}(x,y) = \text{eval}(x,z)\right]$.
\end{enumerate}
\item The relations $R_{0}$, $R_{1}$ satisfy:
\begin{enumerate}
\item $R_{0} \subseteq O_{0} \times O_{1}$.
\item $R_{1} \subseteq F \times O_{1}$.
\item $(\forall x \in F)(\forall z \in O_{1})\left[ R_{0}(\text{eval}(x,z),z) \leftrightarrow R_{1}(x,z)\right]$.
\end{enumerate}
\end{enumerate}

Define $\mathbb{K}$ to be the class of finite models of $T$.  

\begin{lem} \label{NSOP1 Fraisse class}
The class $\mathbb{K}$ is a Fra\"iss\'e class.  Moreover, it is uniformly locally finite.  
\end{lem}

\begin{proof}
HP is clear as the axioms of $T$ are universal.  The argument for JEP is identical to that for SAP, so we show SAP.  Suppose $A,B,C \in \mathbb{K}$ where $A \subseteq B,C$ and $B \cap C = A$.  It suffices to define a $L$-structure with domain $D = B \cup C$, extending both $B$ and $C$.  First, note that if $F^{B}$ is non-empty, then every $E^{B}$-class intersects $O^{B}_{0}$, but if $F^{B} = \emptyset$, it is possible that there are $E^{B}$-equivalence classes disjoint from $O_{0}^{B}$.  In this latter case, we can extend $B$ to $B'$ so that each equivalence class contains an element of $O_{0}$:  Let $(K_{i})_{i < l}$ list the $E^{B}$-classes $K$ of $B$ such that $O_{0}^{B} \cap K = \emptyset$.  Let $B'$ be the $L$-structure with underlying set $B \cup \{*_{i} : i < l\}$ where the $*_{i}$ are new formal elements.  Consider $B'$ as an $L$-structure via the the following interpretations: for the unary predicates, interpret $F^{B'} = F^{B} = \emptyset$, $O_{0}^{B'} = O_{0}^{B} \cup \{*_{i} : i < l\}$, $O_{1}^{B'} = O_{1}^{B}$, and $O^{B'} = O_{0}^{B'} \cup O_{1}^{B'}$.  Let $R_{0}^{B'} = R_{0}^{B}$, $R_{1}^{B'} = R_{1}^{B}$, and let $E^{B'}$ be the equivalence relation generated by $E^{B} \cup \{(b,*_{i}) : i < l, b \in K_{i}\}$.  As $F^{B} = F^{B'} = \emptyset$, we can only define $\text{eval}^{B'} : F^{B'} \times O^{B'} \to O_{0}^{B'}$ to be the empty function.  It is clear that $B'$ is in $\mathbb{K}$, extends $B$, and every equivalence class not represented by an element of $A$ contains an element of $O_{0}$.  By a symmetric argument, we may also extend $C$ to $C'$ so that every $E^{C}$-class not represented by an element of $A$ contains an element of $O_{0}^{C'}$.  Replacing $B$ and $C$ by $B'$ and $C'$ respectively, we may assume that all classes of $B$ and $C$ are either represented by an element of $A$ or by an element of $O_{0}^{B}$ or $O_{0}^{C}$ respectively.  

Now we describe the construction of $D$.  Interpret $O_{0}^{D}$, $O_{1}^{D}$, and $F^{D}$ by $O_{i}^{D} = O_{i}^{B} \cup O_{i}^{C}$ for $i = 0,1$, $O^{D} = O_{0}^{D} \cup O_{1}^{D}$, and $F^{D} = F^{B} \cup F^{C}$.  Let $E^{D}$ be the equivalence relation generated by $E^{B} \cup E^{C}$.  It follows that if $b \in B$, $c \in C$ and $(b,c) \in E^{D}$, then there is some $a \in A$ so that $(a,b) \in E^{B}$ and $(a,c) \in E^{C}$ and, moreover, $(O^{D}, E^{D})$ extends both $(O^{B}, E^{B})$ and $(O^{C}, E^{C})$ as equivalence relations.  Put $R_{0}^{D} = R_{0}^{B} \cup R_{0}^{C}$.  

Next we define the interpretation $\text{eval}^{D}$.  Let $\{a_{i} : i < k_{0}\}$ enumerate a collection of representatives for the $E^{A}$-classes in $A$.  Then let $\{b_{i} : i < k_{1}\}$ and $\{c_{i} : i < k_{2}\}$ enumerate representatives for the $E^{B}$- and $E^{C}$-classes of elements not represented by an element of $A$, respectively.  By the remarks above, we may assume each $b_{i}$ and $c_{i}$ are in $O_{0}^{D}$.  Then every element of $O^{D}$ is equivalent to a unique element of 
$$
X = \{a_{i} : i < k_{0}\} \cup \{b_{i} : i < k_{1}\} \cup \{c_{i} : i < k_{2}\}.  
$$
Suppose $d \in X$.  If $f \in F^{A}$, define $\mathrm{eval}^{D}(f,d) = \mathrm{eval}^{B}(f,d)$ if $d \in B$ and $\mathrm{eval}^{D}(f,d) = \mathrm{eval}^{C}(f,d)$ if $d \in C$, which is well-defined as $A$ is a substructure of both $B$ and $C$.  If $f \in F^{B} \setminus F^{A}$, define $\text{eval}^{D}(f,d) = \text{eval}^{B}(f,d)$ if $d \in B$ and $\text{eval}^{D}(f,d) = d$ otherwise.  Likewise, if $f \in F^{C} \setminus F^{A}$, put $\text{eval}^{D}(f,d) = \text{eval}^{C}(f,d)$ if $d \in C$ and $\text{eval}^{C}(f,c) = c$ otherwise.  This defines $\text{eval}$ on $F^{D} \times X$.  More generally, if $f \in F^{D}$ and $e \in O^{D}$, define $\text{eval}^{D}(f,e) = \text{eval}^{D}(f,d)$ for the unique $d \in X$ equivalent to $e$.  

To complete the construction, we must describe the interpretation of $R_{1}^{D}$.  Put 
$$
R_{1}^{D} = R_{1}^{B} \cup R_{1}^{C} \cup \{(f,d) \in F^{D} \times O_{1}^{D} : (\text{eval}^{D}(f,d),d) \in R_{0}^{D}\}.
$$  
We check that this defines an extension of $B$ and $C$.  If $b \in F^{B}$, $b' \in O_{1}^{B}$, and $(\text{eval}^{D}(b,b'),b') \in R_{0}^{D}$, then $(\text{eval}^{D}(b,b'),b') \in R_{0}^{B}$ and $\text{eval}^{D}(b,b') = \text{eval}^{B}(b,b')$ so $(\text{eval}^{B}(b,b'),b') \in R_{0}^{B}$ and therefore $R_{1}^{B}(b,b')$.  This shows $R_{1}^{D} \upharpoonright B = R_{1}^{B}$.  Likewise $R_{1}^{D} \upharpoonright C = R_{1}^{C}$.  Therefore $D$ extends $B$ and $C$.  

Now to conclude we must show $D \in \mathbb{K}$.  It is clear that $D$ satisfies axioms (1)-(3), so we are left with checking (4).  Suppose $(f,d) \in F^{D} \times O_{1}^{D} \setminus (F^{B} \times O_{1}^{B} \cup F^{C} \times O_{1}^{C})$ and $d' = \text{eval}^{D}(f,d)$.  Then, by definition, if $(d,d') \in R_{0}^{D}$, then $(f,d) \in R_{1}^{D}$.  On the other hand, if $(f,d) \in R_{1}^{D}$ then, because $(f,d) \not\in R_{1}^{B} \cup R_{1}^{C}$, we must have $(d,d') \in R_{0}^{D}$, again by the definition of $R_{1}^{D}$.  It is clear that if $(f,d) \in F^{B} \times O_{1}^{B} \cup F^{C} \times O_{1}^{C}$ then $(f,d) \in R_{0}^{D}$ if and only if $(f,d) \in R_{1}^{D}$ because $D$ extends $B$ and $C$ which are in $\mathbb{K}$.  Therefore $D$ satisfies axiom (4) which shows $D \in \mathbb{K}$.  This shows $\mathbb{K}$ has the amalgamation property.  

Finally, note that a structure in $\mathbb{K}$ generated by $k$ elements is obtained by applying $\leq k$ functions of the form $\text{eval}(f,-)$ to $\leq k$ elements in $O$, so has cardinality $\leq k^{2}+k$.  This shows $\mathbb{K}$ is uniformly locally finite.   
\end{proof}

\begin{cor}
$T$ has a model completion $T^{*}$ which is the theory of the Fra\"iss\'e limit of $\mathbb{K}$.  The theory $T^{*}$ eliminates quantifiers and is $\aleph_{0}$-categorical.  
\end{cor}

We will write $\mathbb{M} \models T^{*}$ for a monster model of $T^{*}$.  We will now show that $T^{*}$ is NSOP$_{1}$ by appealing to the following criterion:  

\begin{fact} \label{criterion} \cite[Proposition 5.8]{ArtemNick} 
Assume there is an \(\text{Aut}(\mathbb{M})\)-invariant ternary relation \(\ind\) on small subsets of \(\mathbb{M}\) satisfying the following properties, for an arbitrary $M\prec \mathbb{M}$ and arbitrary tuples from $\mathbb{M}$:
\begin{enumerate}
\item Strong finite character:  if \(a \nind_{M} b\), then there is a formula \(\varphi(x,b,m) \in \text{tp}(a/Mb)\) such that for any \(a' \models \varphi(x,b,m)\), \(a' \nind_{M} b\).
\item Existence over models: \(a \ind_{M} M\).
\item Monotonicity: if \(aa' \ind_{M} bb'\), then \(a \ind_{M} b\).
\item Symmetry: if \(a \ind_{M} b\), then \(b \ind_{M} a\).
\item The independence theorem: if \(a \ind_{M} b\), \(a' \ind_{M} c\), \(b \ind_{M} c\) and $a \equiv_{M} a'$, then there exists $a''$ with $a'' \equiv_{Mb} a$, $a'' \equiv_{Mc} a'$, and $a'' \ind_{M} bc$.
\end{enumerate}
Then \(T\) is NSOP\(_{1}\).
\end{fact}

\begin{defn}
Define a ternary relation $\ind^{*}$ on small subsets of $\mathbb{M}$ by:  $a \ind^{*}_{C} b$ if and only if
\begin{enumerate}
\item $\text{dcl}(aC)/E \cap \text{dcl}(bC) /E \subseteq \text{dcl}(C)/E$.
\item $\text{dcl}(aC) \cap \text{dcl}(bC) \subseteq \text{dcl}(C)$.  
\end{enumerate}
where $X/E = \{[x]_{E} :x \in X\}$ denotes the collection of $E$-classes represented by an element of $X$.  
\end{defn}

\begin{lem}\label{independencetheoremfor2functions}
The relation $\ind^{*}$ satisfies the independence theorem over models:  if $M \models T^{*}$, $a \equiv_{M} a'$, and, additionally, $a \ind^{*}_{M} B$, $a' \ind^{*}_{M} C$ and $B \ind^{*}_{M} C$ then there is $a''$ with $a'' \equiv_{MB} a$, $a'' \equiv_{MC} a'$, and $a'' \ind^{*}_{M} BC$.  
\end{lem}

\begin{proof}
Without loss of generality, we may assume that $M \subseteq B,C$, and that $B$ and $C$ are definably closed.  Write $a = (d_{0},\ldots, d_{k-1},e_{0},\ldots, e_{l-1}, f_{0},\ldots, f_{m-1})$ with $d_{i} \in F$,  $e_{j} \in O_{0}$, $f_{k} \in O_{1}$, and likewise 
\noindent $a' = (d'_{0}, \ldots, d'_{k-1}, e'_{0}, \ldots, e'_{l-1}, f'_{0}, \ldots, f'_{m-1})$.  
Fix an automorphism $\sigma \in \text{Aut}(\mathbb{M}/M)$ with $\sigma(a) = a'$.  Let $U = \{u_{g} : g \in \text{dcl}(aB) \setminus B\}$ and $V = \{v_{g} : g \in \text{dcl}(a'C) \setminus C\}$ denote collection of new formal elements with $u_{g} = v_{\sigma(g)}$ for all $g \in \langle aM \rangle \setminus B$.  Let, then, $a_{*}$ be defined as follows:  
\begin{eqnarray*}
a_{*} &=& (u_{d_{0}}, \ldots, u_{d_{k-1}}, u_{e_{0}}, \ldots, u_{e_{l-1}}, u_{f_{0}},\ldots, u_{f_{m-1}}) \\
&=& (v_{d'_{0}}, \ldots, v_{d'_{k-1}}, v_{e'_{0}}, \ldots, v_{e'_{l-1}}, v_{f_{0}'},\ldots, v_{f_{m-1}'}).
\end{eqnarray*}
We will construct by hand an $L$-structure $D$ extending $\langle BC \rangle$ with domain $UV\langle BC \rangle$ in which $a_{*} \equiv_{B} a$, $a_{*} \equiv_{C} a'$ and $a_{*} \ind^{*}_{M} BC$. 

There is a bijection $\iota_{0}: \text{dcl}(aB) \to BU$ given by $\iota_{0}(b) = b$ for all $b \in B$ and $\iota_{0}(g) = u_{g}$ for all $g \in \text{dcl}(aB) \setminus B$.  Likewise, we have a bijection $\iota_{1}: \text{dcl}(a'C) \to CV$ given by $\iota_{1}(c) = c$ for all $c \in C$ and $\iota_{1}(g) = v_{g}$ for all $g \in \text{dcl}(a'C) \setminus C$.  The union of the images of these functions is the domain of the structure $D$ to be constructed and their intersection is $\iota_{0}(\langle aM \rangle) = \iota_{1}(\langle a'M \rangle)$.  Consider $BU$ and $CV$ as $L$-structures by pushing forward the structure on $\text{dcl}(aB)$ and $\text{dcl}(a'C)$ along $\iota_{0}$ and $\iota_{1}$, respectively.  Note that $\iota_{0}|_{\langle aM \rangle} = (\iota_{1} \circ \sigma)|_{\langle aM \rangle}$.    

We are left to show that we can define an $L$-structure on $UV\langle BC \rangle$ extending that of $BU$, $CV$, and $\langle BC \rangle$ in such a way as to obtain a model of $T$.  To begin, interpret the predicates by $O_{i}^{D} = O_{i}^{BU} \cup O_{i}^{CV} \cup O_{i}^{\langle BC \rangle}$ for $i = 0,1$, $O^{D} = O_{0}^{D} \cup O_{1}^{D}$, $F^{D} = F^{BU} \cup F^{CV} \cup F^{\langle BC \rangle}$, and $R_{0}^{D} = R_{0}^{BU} \cup R_{0}^{CV} \cup R_{0}^{\langle BC \rangle}$.  Let $E^{D}$ be defined to be the equivalence relation generated by $E^{BU}$, $E^{CV}$, and $E^{\langle BC \rangle}$.  The interpretation of the predicates defines extensions of the given structures since if $g$ is an element of $\iota_{0}(\langle aM \rangle) = \iota_{1}(\langle a'M \rangle)$ then $\iota_{0}^{-1}(g)$ is in the predicate $O$ if and only if $\iota_{1}^{-1}(g)$ is as well, and, moreover, it is easy to check that our assumptions on $a,a',B,C$ entail that no pair of inequivalent elements in $BU$, $CV$, or $\langle BC \rangle$ become equivalent in $D$.  

Next we define the function $\text{eval}^{D}$ extending $\text{eval}^{BU} \cup \text{eval}^{CV} \cup \text{eval}^{\langle BC \rangle}$. We first claim that $\text{eval}^{BU} \cup \text{eval}^{CV} \cup \text{eval}^{\langle BC \rangle}$ is a function.  The intersection of the domains of the first two functions is $\iota_{0}(\langle aM \rangle) = \iota_{1}(\langle aM \rangle)$.  If $b,b'$ are in this intersection, we must show
$$
\text{eval}^{BU}(b,b') = c \iff \text{eval}^{CV}(b,b') = c.
$$
Choose $b_{0},b'_{0}, c_{0} \in \langle a M \rangle$ and $b_{1},b'_{1},c_{1} \in \langle a'M \rangle$ with $\iota_{i}(b_{i},b'_{i},c_{i}) = (b,b',c)$ for $i = 0,1$.  Then since $\iota_{0} = \iota_{1} \circ \sigma$ on $\langle aM \rangle$, we have 
\begin{eqnarray*}
\mathbb{M} \models \text{eval}(b_{0},b'_{0}) = c_{0} &\iff& \mathbb{M} \models \text{eval}(\sigma(b_{0}),\sigma(b'_{0})) = \sigma(c_{0}) \\
&\iff& \mathbb{M} \models \text{eval}(b_{1},b'_{1}) = c_{1}.
\end{eqnarray*}
Since $\text{eval}^{BU}$ and $\text{eval}^{CV}$ are defined by pushing forward the structure on $\langle aB\rangle$ and $\langle a'C \rangle$ along $\iota_{0}$ and $\iota_{1}$, respectively, this shows that $\text{eval}^{BU} \cup \text{eval}^{CV}$ defines a function.  Now the intersection of $\langle BC \rangle$ with $BU \cup CV$ is $BC$ and, by construction, all 3 functions agree on this set.  So the union defines a function.

Note that because $BU$, $CV$, and $\langle BC \rangle$ all contain a model $M$ and therefore have non-empty $F$-sort, every $E^{D}$ class is represented by an element of $O_{0}^{D}$.  Choose a complete set of $E^{D}$-class representatives $\{d_{i}: i < \alpha\}$ so that if $d_{i}$ represents an $E^{D}$-class that meets $M$ then $d_{i} \in M$ and $d_{i} \in O_{0}$.  If $e \in O^{D}$ is $E^{D}$-equivalent to some $e'$ and $(f,e')$ is in the domain of $\text{eval}^{BU} \cup \text{eval}^{CV} \cup \text{eval}^{\langle BC \rangle}$, define $\text{eval}^{D}(f,e)$ to be the value that this function takes on $(f,e')$.  On the other hand, if $f \in F^{D} \setminus (F^{BU} \cup F^{CV} \cup F^{\langle BC \rangle})$ or $e$ is not $E^{D}$-equivalent to any element on which $\text{eval}^{D}(f,-)$ has already been defined, put $\text{eval}^{D}(f,e) = d_{i}$ for the unique $d_{i}$ which is $E^{D}$-equivalent to $e$.  This now defines $\text{eval}^{D}$ on all of $F^{D} \times O^{D}$ and, by construction, $\text{eval}^{D}(f,-)$ is a selector function for $E^{D}$ for all $f \in F^{D}$.  

To conclude, we must interpret $R_{1}$ on $D$.  In order to build a structure that satisfies axiom (4), we are forced to interpret 
$$
R^{D}_{1} = \{(f,b) \in F \times O_{1} : (\text{eval}^{D}(f,b),b) \in R_{0}^{D}\}.
$$
In order to ensure that $D$ is an extension of $BU$, $CV$, and $\langle BC \rangle$, we have show that for all $X \in \{BU, CV, \langle BC \rangle\}$, $R_{1}^{D} \upharpoonright X = R_{1}^{X}$.  Suppose we have $f,a,b \in X$ with $\text{eval}^{X}(f,b) = a$.  Then because $X$ is a model of $T$, we have $R_{0}^{X}(a,b) \iff R_{1}^{X}(f,b)$ and, by construction, $R_{0}^{X}(a,b) \iff R_{0}^{D}(a,b)$.  By definition, $R_{1}^{D}(f,b) \iff R_{0}^{D}(a,b)$.  This shows $R_{1}^{D}(f,b) \iff R_{1}^{X}(f,b)$, hence $R_{1}^{D} \upharpoonright X = R^{X}_{1}$.  

We have already argued that $BU$ and $CV$ are substructures of $D$ - it follows that every $E^{D}$-class represented by an element of $a_{*}$ can only be equivalent to an element of $B$ or $C$ if it is equivalent to an element of $M$.  Moreover, our construction has guaranteed that $\langle a_{*}M \rangle^{D} \cap \langle BC \rangle \subseteq BU \cap \langle BC \rangle^{D} \subseteq B$ and, by similar reasoning, $\langle a_{*}M \rangle \cap \langle BC \rangle \subseteq C$.  This implies $\langle a_{*} M \rangle^{D} \cap \langle BC \rangle B \cap C \subseteq M$, so $a_{*} \ind^{*}_{M} BC$.  Embedding $D$ into $\mathbb{M}$ over $\langle BC \rangle$, we conclude. 
\end{proof}

\begin{cor}
The theory $T^{*}$ is NSOP$_{1}$.
\end{cor}

\begin{proof}
The relation $\ind^{*}$ is easily seen to satisfy properties (1) through (4) from Fact \ref{criterion} and the independence theorem is established in Lemma \ref{independencetheoremfor2functions}.  This implies $T^{*}$ is NSOP$_{1}$.  
\end{proof}

\begin{rem}
One may additionally show that $\ind^{*} = \ind^{K}$ over models.  As we won't need Kim-independence in what follows, we omit the proof.  
\end{rem}

\begin{prop}
$\kappa_{\mathrm{shred}}(T^{*}) = \infty$.
\end{prop}

\begin{proof}
Let $\kappa$ be an arbitrary regular cardinal.  Inductively, we may choose a sequence of elements $\langle a_{i} : i < \kappa\rangle$ and a sequence of sequences $\langle \overline{b}_{i} : i < \kappa \rangle$ so that 
\begin{enumerate}
\item For all $i < \kappa$, $a_{i} \in O_{0}$.  
\item For all $i < \kappa$, $\overline{b}_{i} = \langle b_{i,j} : j < \omega  \rangle$ is an $a_{<i}\overline{b}_{<i}$-indiscernible sequence of elements of $O_{1}$ in the same $E$-class as $a_{i}$, with $R_{0}(a_{i},b_{i,j})$ if and only if $j$ is even. 
\end{enumerate}

Let $p(x) = \{\mathrm{eval}(x;a_{i}) = a_{i} : i < \kappa\}$ and fix some $i < \kappa$.  Because each $b_{i,j}$ is $E$-equivalent to $a_{i}$ and $\mathrm{eval}(x,-)$ is a selector function, $\mathrm{eval}(x,a_{i}) = a_{i}$ implies $\mathrm{eval}(x;b_{i,j}) = a_{i}$.  It follows from axiom 4(c) of $T$ that $\mathrm{eval}(x,a_{i}) = a_{i}$ implies $R_{0}(a_{i},b_{i,j}) \leftrightarrow R_{1}(x,b_{i,j})$ for all $j$.  Therefore, $\mathrm{eval}(x,a_{i}) = a_{i} \vdash R_{1}(x;b_{i,j})$ if $j$ is even and $\mathrm{eval}(x,a_{i}) = a_{i} \vdash \neg R_{1}(x;b_{i,j})$ if $j$ is odd.  This shows $\mathrm{eval}(x;a_{i}) = a_{i} \in p\upharpoonright a_{<i+1}$ explicitly shreds over $a_{<i}$.  Since $\kappa$ is arbitrary, we conclude $\kappa_{\mathrm{shred}}(T^{*}) = \infty$.  
\end{proof}

\subsection{An NTP$_{2}$ example}

In this subsection, we describe an NTP$_{2}$ example with $\kappa_{\mathrm{shred}}(T) = \infty$.  Recall the definition of NTP$_{2}$ theories:
\begin{defn}
A formula $\varphi(x;y)$ has \emph{the tree property of the second kind} (TP$_{2}$) if there is an array of tuples $(a_{i,j})_{i,j < \omega}$ and $k < \omega$ satisfying the following:
\begin{enumerate}
\item For all $f : \omega \to \omega$, $\{\varphi(x;a_{i,f(i)}) : i < \omega\}$ is consistent.
\item For all $i < \omega$, $\{\varphi(x;a_{i,j}) : j < \omega\}$ is $k$-inconsistent.  
\end{enumerate}
A theory is said to have TP$_{2}$ if some formula has TP$_{2}$ modulo $T$ and is otherwise called NTP$_{2}$.  
\end{defn}
The class of NTP$_{2}$ contains both the NIP and simple theories, so it is natural to ask if NTP$_{2}$ implies $\kappa_{\mathrm{shred}}(T) < \infty$ but we show this is not the case.  

The following fact will be useful in checking that the theory we construct is NTP$_{2}$:
\begin{fact} \label{NTP2 facts}
\begin{enumerate}
\item If $T$ has TP$_{2}$, there is a formula $\varphi(x;y)$ witnessing this with $l(x) = 1$ \cite[Corollary 2.9]{ChernikovNTP2}.
\item If $\varphi(x;y)$ has TP$_{2}$, then this will be witnessed with respect to an array of parameters $(a_{i,j})_{i,j < \omega}$ that is \emph{mutually indiscernible}\textemdash that is, $\overline{a}_{i}$ is $\overline{a}_{\neq i}$-indiscernible for all $i < \omega$ \cite[Lemma 2.2]{ChernikovNTP2}.  
\end{enumerate}
\end{fact}

Let $L$ a language consisting of two binary relations $R, \unlhd$, and a binary function $\wedge$ and the sublanguage consisting of just $\unlhd$ and $\wedge$ is $L_{\mathrm{tr}}$.  The class $\mathbb{K}$ will consist of finite $L$-structures $(A,\unlhd^{A}, \wedge^{A}, R^{A})$ so that $(A,\unlhd^{A},\wedge^{A})$ is a meet-tree where $\wedge^{A}$ is the meet function, and $R^{A}$ is a graph on $A$.  Denote the class of finite $\wedge$-trees $(A, \unlhd^{A}, \wedge^{A})$ by $\mathbb{K}_{0}$.  This is a Fra\"iss\'e class with the strong amalgamation property (SAP) and the theory $T_{\mathrm{tr}}$ of its Fra\"iss\'e limit is dp-minimal \cite[Exercise 2.50, Example 4.28]{simon2015guide}, which means given a mutually indiscernible array $(a_{i,j})_{i < 2, j < \omega}$ and element $c$, there is some $i  < 2$ such that $\overline{a}_{i}$ is $c$-indiscernible.   
%

\begin{lem} \label{its a fraisse class}
The class $\mathbb{K}$ is a Fra\"iss\'e class.  Moreover, the reduct of the Fra\"iss\'e limit of $\mathbb{K}$ to $L_{\mathrm{tr}}$ is the Fra\"iss\'e limit of $\mathbb{K}_{0}$.  
\end{lem}

\begin{proof}
HP is clear and JEP will follow from a similar argument to SAP, so we will prove SAP.  Fix $\tilde{A}, \tilde{B}_{0},\tilde{B}_{1} \in \mathbb{K}$ such that $\tilde{A}$ is an $L$-substructure of both $\tilde{B}_{0}$ and $\tilde{B}_{1}$ and $\tilde{B}_{0} \cap \tilde{B}_{1} = \tilde{A}$.  Let $A = \tilde{A} \upharpoonright L_{\mathrm{tr}}$ and $B_{i} = \tilde{B}_{i} \upharpoonright L_{\mathrm{tr}}$ for $i = 0,1$.  By SAP in $\mathbb{K}_{0}$, there is $D \in \mathbb{F}$ extending both $B_{0}$ and $B_{1}$.  We may expand $D$ to an $L$-structure $\tilde{D}$ by setting $R^{\tilde{D}} = R^{\tilde{B}_{0}} \cup R^{\tilde{B}_{1}}$.  This establishes SAP for $\mathbb{K}$.    	

Next, suppose $A, B \in \mathbb{K}_{0}$ and $\pi: A \to B$ is an $L_{tr}$-embedding.  If $\tilde{A} \in \mathbb{K}$ is an expansion of $A$, then we can expand $B$ to the $L$-structure $\tilde{B}$ in which $R^{\tilde{B}} = \{(\pi(a),\pi(a')) : (a,a') \in R^{\tilde{A}}\}$.  Clearly we have $\tilde{B} \in \mathbb{K}$ and $\pi$ is also an $L$-embedding so by \cite[Lemma 2.8]{ramsey2015invariants}, the reduct of the Fra\"iss\'e limit of $\mathbb{K}$ is the Fra\"iss\'e limit of $\mathbb{K}_{0}$.  
\end{proof}

By Lemma \ref{its a fraisse class}, we know that $\mathbb{K}$ has a Fra\"iss\'e limit which is an $\omega$-categorical expansion of $T_{\mathrm{tr}}$ by a (random) graph.  Let $T$ denote its theory and let $\mathbb{M}$ and $\mathbb{M}_{\mathrm{tr}}$ denote the monster models of $T$ and $T_{\mathrm{tr}}$ respectively.  

\begin{lem} \label{sequence stuff}
Suppose we are given an $L$-indiscernible sequence $I = \langle a_{i} : i \in \mathbb{Z} \rangle$ and an element $b$ so that $I$ is $L_{\mathrm{tr}}$-indiscernible over $b$.  Then there is $b' \equiv^{L}_{a_{0}} b$ so that $I$ is $L$-indiscernible over $b'$.
\end{lem}

\begin{proof}
Let $\sigma \in \mathrm{Aut}_{L_{\mathrm{tr}}}(\mathbb{M}/b)$ be an automorphism so that $\sigma(a_{i}) = a_{i+1}$.  Let $B$ denote the $L$-structure generated by $\langle a_{i} : i \in \mathbb{Z} \rangle$ and let $A_{0}$ be the $L$-structure generated by $a_{0}b$.  Now expand the $L_{\mathrm{tr}}$-structure $\langle b(a_{i})_{i \in \mathbb{Z}} \rangle_{L_{\mathrm{tr}}}$ to an $L$-structure $M$ by setting
$$
R^{M} =  R^{B} \cup \bigcup_{i \in \mathbb{Z}} \sigma^{i}(R^{A_{0}}).
$$
\textbf{Claim 1}:  If $i \in \mathbb{Z}$ and $c,d \in B \cap \sigma^{i}(A_{0})$, then $(c,d) \in R^{B}$ if and only if $(c,d) \in \sigma^{i}(R^{A_{0}})$.  

\emph{Proof of claim}:  This is clear if $i = 0$, since $R^{B} = R^{\mathbb{M}} \upharpoonright B$ and $R^{A_{0}} = R^{\mathbb{M}} \upharpoonright A_{0}$.  In general, if $c,d \in B \cap \sigma^{i}(A_{0})$, there are $L_{\mathrm{tr}}$-terms $t,t',s,s'$ so that 
\begin{eqnarray*}
c &=& t(a_{<i},a_{i},a_{>i}) = t'(b,a_{i}) \\
d &=& s(a_{<i},a_{i},a_{>i}) = s'(b,a_{i}).  
\end{eqnarray*}
By indiscernibility, it follows that if $\sigma^{i}(c',d') = (c,d)$, then we have
\begin{eqnarray*}
c' &=& t(a_{<0},a_{0},a_{>0}) = t'(b,a_{0}) \\
d' &=& s(a_{<0},a_{0},a_{>0}) = s'(b,a_{0}),
\end{eqnarray*} 
and we know that $(c',d') \in R^{B}$ if and only if $(c',d') \in R^{A_{0}}$, by the $i = 0$ case.  By indiscernibility, $(c',d') \in R^{B}$ if and only if $(c,d) \in R^{B}$ and hence $(c,d) \in R^{B}$ if and only if $(c,d) \in \sigma^{i}(R^{A_{0}})$.  \qed

\noindent \textbf{Claim 2}:  If $i >0$ and $c,d \in A_{0} \cap \sigma^{i}(A_{0})$ then $(c,d) \in R^{A_{0}}$ if and only if $(c,d) \in \sigma^{i}(R^{A_{0}})$.  

\emph{Proof of claim}:  As in the proof of the previous claim, there are $L_{\mathrm{tr}}$-terms $t,t',s$, and $s'$ so that we have the following equalities:
\begin{eqnarray*}
c &=& t(a_{0},b) = t'(a_{i},b) \\
d &=& s(a_{0},b) = s'(a_{i},b).
\end{eqnarray*}  
Then by $L_{\mathrm{tr}}$-indiscernibility over $b$, we have also $t(a_{0},b) = t'(a_{i+1},b)$ and $t(a_{1},b) = t'(a_{i+1},b)$, hence $t(a_{0},b) = t(a_{1},b)$.  Likewise, we have $s(a_{0},b) = s(a_{1},b)$.  In particular, this shows $\sigma(c,d) = (c,d)$ so the claim follows.  \qed

Now, by Claim 1, it follows that for all $c,d \in B$, we have $(c,d) \in R^{M}$ if and only if $(c,d) \in R^{B}$, so $M$ extends $B$.  Likewise, by Claim 2, $M$ extends $A_{0}$ and $\sigma^{i}$ induces an $L$-isomorphism of $A_{0}$ and the structure generated by $ba_{i}$ in $M$, for all $i \in \mathbb{Z}$.  Embed $M$ into $\mathbb{M}$ over $B$ and let $b'$ be the image of $b$ under this embedding.  Then by quantifier-elimination, $a_{0}b \equiv a_{i}b'$ for all $i \in \mathbb{Z}$.  After applying Ramsey, compactness, and an automorphism, we can find $b'' \equiv_{a_{0}} b'$ so that $I$ is $L$-indiscernible over $b''$, completing the proof.  
\end{proof}

\begin{cor}
The theory $T$ is NTP$_{2}$ (and is, in fact, inp-minimal).
\end{cor}

\begin{proof}
If $T$ has TP$_{2}$, then, by Fact \ref{NTP2 facts} and compactness, there is an $L$-formula $\varphi(x;y)$ with $l(x) = 1$ that witnesses TP$_{2}$ with repect to the mutually indiscernible array $(a_{i,j})_{i < \omega, j \in \mathbb{Z}}$.  Let $b \models \{\varphi(x;a_{i,0}) : i < \omega\}$.  As $T_{\mathrm{tr}}$ is dp-minimal, there is a row $i = 0$ or $i = 1$ so that $\langle a_{i,j} : j \in \mathbb{Z} \rangle$ is $b$-indiscernible in the language $L_{\mathrm{tr}}$.  By Lemma \ref{sequence stuff}, there is $b' \equiv^{L}_{a_{i,0}} b$ such that $\langle a_{i,j} : j \in \mathbb{Z} \rangle$ is $b'$-indiscernible in the language $L$.  Then $b' \models \{\varphi(x;a_{i,j}) : j \in \mathbb{Z} \}$, contradicting the row-wise inconsistency required for TP$_{2}$.  
\end{proof}

\begin{prop}
$\kappa_{\mathrm{shred}}(T) = \infty$.  
\end{prop}

\begin{proof}
Let $\kappa$ be an arbitrary regular cardinal.  Inductively, we may choose a sequence of elements $\langle a_{i} : i < \kappa\rangle$ and a sequence of sequences $\langle \overline{b}_{i} : i < \kappa \rangle$ so that 
\begin{enumerate}
\item For all $i < \kappa$, $\overline{b}_{i} = \langle b_{i,j} : j < \omega \rangle$ is an $a_{<i}\overline{b}_{<i}$-indiscernible sequence of pairwise incomparable elements, incomparable with $a_{i}$, with $b_{i,j} \wedge b_{i,j'} = a_{i} \wedge b_{i,j}$ for all $j \neq j'$ and $R(a_{i}\wedge b_{i,j},b_{i,j})$ if and only if $j$ is even.
\item If $i < i' < \kappa$, then $a_{i} \vartriangleleft a_{i'} \wedge b_{i',j}$ for all $j$. 
\end{enumerate}

There is no problem continuing the induction, since $T$ is the generic $\wedge$-tree with a random graph.  

\begin{center}
\includegraphics[scale=.09]{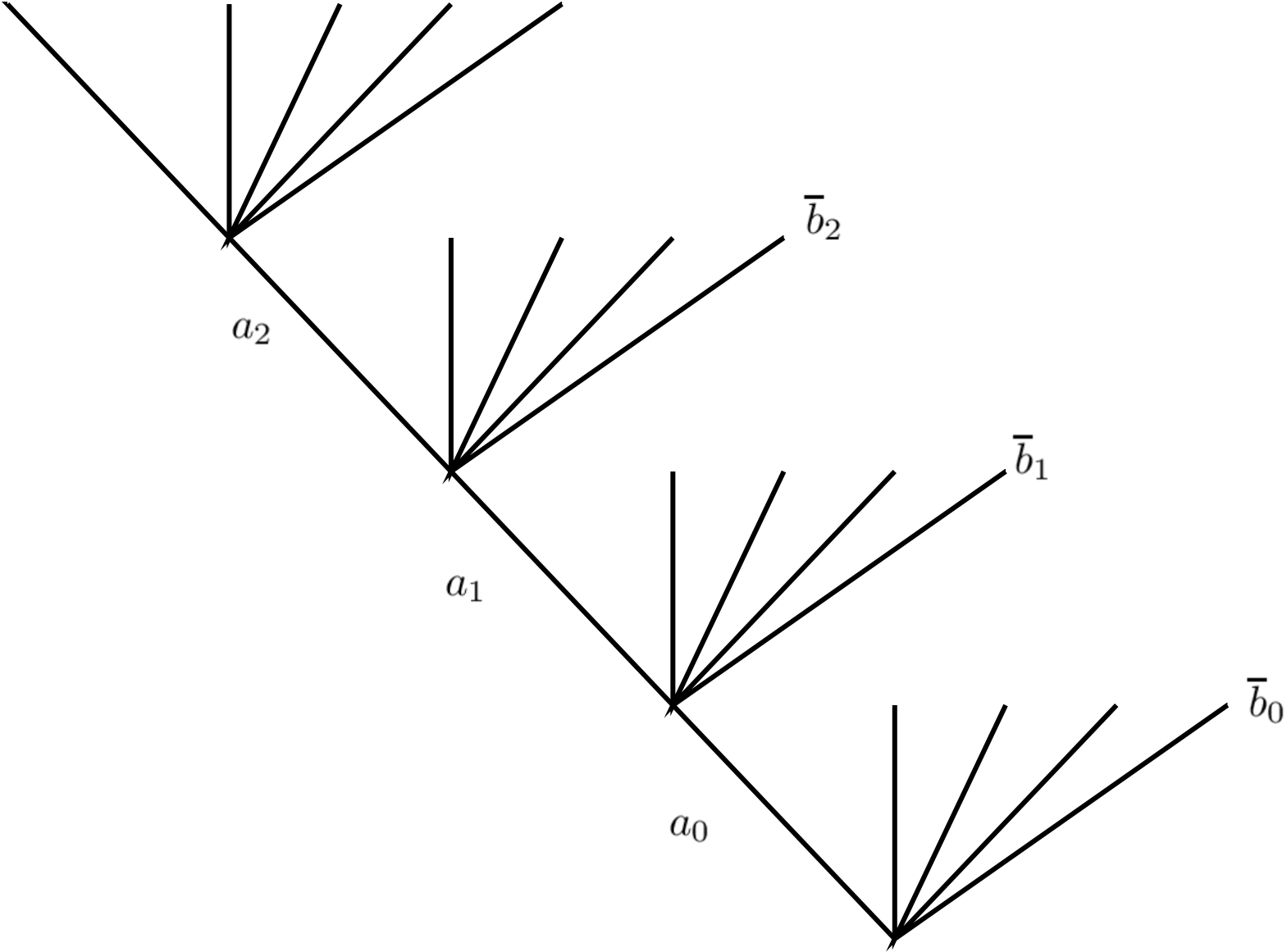} 
\captionof{figure}{Illustration of the choice of $a_{i}$ and $\overline{b}_{i}$}
\end{center}

Let $p(x) = \{x \unrhd a_{i} : i < \kappa \}$.  Notice that if $x \unrhd a_{i}$, then $x \wedge b_{i,j} = a_{i} \wedge b_{i,j}$ and hence $x \unrhd a_{i} \vdash R(x \wedge b_{i,j}, b_{i,j})$ if and only if $j$ is even.  It follows that the formula $x \unrhd a_{i}$ explicitly shreds over $a_{<i}$.  As $\kappa$ is arbitrary, $\kappa_{\mathrm{shred}}(T)  = \infty$.  
\end{proof}

\section{A criterion for singular compactness}

In this section, we give a sufficient condition for having singular compactness, which is the negation of exact saturation (Definition \ref{singular compactness} below).  If $\Delta\left(x,y\right)$ is a set of formulas over $\emptyset$, then a (partial) \emph{$\Delta$-type} is a consistent set of instances of formulas from $\Delta$. We may refer to a $\{\varphi\}$-type as a $\varphi$-type. It is important to note that by a $\varphi$\emph{-type} we mean a consistent set of \emph{positive} instances of $\varphi$, and do not include instances of $\neg \varphi$.  

\begin{defn} \label{singular compactness}
Suppose that $T$ is a complete first order theory and $\Delta$ is a set of formulas over $\emptyset$.
Say that $T$ has\emph{ singular compactness} for $\Delta$ if whenever
$M\models T$ is $\mu$-saturated for a singular cardinal $\mu>\left|T\right|$
then $M$ is $\mu^{+},\Delta$-saturated: for every $\Delta$-type
$p$ over a set $A\subseteq M$ with $\left|A\right|\leq\mu$, $p$
is realized in $M$. 
\end{defn}

\begin{condition} \label{cond:for singular compactness} 
For every formula $\varphi\left(x,y\right)$
(perhaps in a fixed set of formulas $\Delta$) there is some formula
$\theta_{\varphi}\left(x,z\right)$ such that for any finite $\varphi$-type
$r\left(x\right)$ over $M\models T$ and every finite set $A\subseteq M^{x}$
of realizations of $r$ there is some $b\in M^{z}$ such that $\theta_{\varphi}\left(A,b\right)$
holds (i.e., $M\models\theta_{\varphi}\left(a,b\right)$ for all $a\in A$)
and $\theta_{\varphi}\left(x,b\right)\vdash r\left(x\right)$. 
\end{condition}

\begin{lem}
\label{lem:sufficient condition}Suppose that $T$ is a complete first
order theory and that Condition \ref{cond:for singular compactness}
holds for $\Delta\left(x,y\right)$. Then $T$ has singular compactness
for $\Delta$. 
\end{lem}

\begin{proof}
Let $p$ be a $\Delta$-type over a set $A$ with $\left|A\right|=\mu$ and suppose $A \subseteq M$, a $\mu$-saturated model of $T$. Write $A=\bigcup_{i<\kappa}A_{i}$
with $\left|A_{i}\right|<\mu$, $\kappa<\mu$. For each $i<\kappa$
find $b_{i}\in M$ such that $b_{i}\models p|_{A_{i}}$ (exists by
$\mu$-saturation). 

By compactness and Condition \ref{cond:for singular compactness},
for each $\varphi\in\Delta$ find $e_{i}^{\varphi}\in\mathbb{M}^{z}$ such
that $\theta_{\varphi}\left(b_{j},e_{i}^{\varphi}\right)$ holds for
all $j\geq i$ and $\theta_{\varphi}\left(x,e_{i}^{\varphi}\right)\vdash\text{tp}^{+}_{\varphi}\left(b_{i}/A_{i}\right)$, the (positive) $\varphi$-type of $b_{i}$ over $A_{i}$.  
 By $\mu$-saturation, find $d_{i}^{\varphi}\in M$ such that $d_{i}^{\varphi}\equiv_{A_{i}\cup\{b_{i} : i<\kappa\}}e_{i}^{\varphi}$.
Then $\{\theta_{\varphi}\left(x,d_{i}^{\varphi}\right): i<\kappa,\varphi\in\Delta\}$
is a type and hence realized in $M$. 
\end{proof}
When does Condition \ref{cond:for singular compactness} hold? If
$T$ is complicated enough, e.g., $T=PA$ or $T=ZFC$, then it holds
since given $\varphi\left(x,y\right)$, we can choose $\theta_{\varphi}\left(x,z\right)=x\in z$.
Indeed, this condition implies that the theory cannot be too tame. 
\begin{prop}
Assume $T$ has infinite models.  If Condition \ref{cond:for singular compactness} holds for every
formula with one variable $x$ then $T$ has TP$_{2}$,
and has SOP$_{n}$ for all $n$.
\end{prop}

\begin{proof}
We start by showing that $T$ has TP$_{2}$. Let $\varphi\left(x,z\right)$
be $\theta_{x\neq y}\left(x,z\right)$. Let $\psi\left(x,w\right)=\theta_{\neg \varphi}\left(x,w\right)$.
We will show that $\xi\left(x,zw\right)=\varphi\left(x,z\right)\land\psi\left(x,w\right)$
witnesses TP$_{2}$. Let $\{a_{i} : i<\omega\}$ be some infinite
set in $\mathbb{M}$. Suppose that $\mathcal{F}$ is an arbitrary family of pairwise disjoint
subsets of $\omega$. It is enough to find some $b_{s}\in M^{zw}$
for every $s\in\mathcal{F}$ such that $\xi\left(a_{i},b_{s}\right)$ holds
whenever $i\in s$, and $\left\{ \xi\left(x,b_{s}\right),\xi\left(x,b_{t}\right)\right\} $
is inconsistent for all $s \neq t$ from $\mathcal{F}$ (see \cite[Lemma 2.19]{kaplan2017some}). By compactness we may assume that $\mathcal{F}$ is finite and
consists of finite sets and replace $\omega$ by some $n<\omega$. 

By choice of $\varphi\left(x,z\right)$ there are $c_{s}$ for $s\in\mathcal{F}$
such that $\varphi\left(a_{i},c_{s}\right)$ holds iff $i\in s$: take the finite type $r_{s}=\{x\neq a_{i} : i\notin s\}$ and $A_{s}=\{a_{i} : i\in s\}$
and apply Condition \ref{cond:for singular compactness}. This already
shows that $T$ has the independence property so
is not NIP. 

We can similarly choose $d_{s}$ by applying Condition \ref{cond:for singular compactness}
for $\varphi$ and taking $r_{s}=\{\neg\varphi\left(x,c_{t}\right) : t\neq s,t\in\mathcal{F}\}$
and $A_{s}=\{a_{i}: i\in s\}$. Then obviously $\xi\left(a_{i},c_{s}d_{s}\right)$
holds if $i\in s$. Also, as $\psi\left(x,d_{s}\right)\vdash\neg\varphi\left(x,c_{t}\right)$
for $t\neq s$, we are done. 

Next we show that $T$ has SOP$_{n}$ for all $n<\omega$.

Let $\varphi_{0}\left(x,y_{0}\right)=\theta_{\neq}\left(x,y_{0}\right)$,
$\varphi_{1}\left(x,y_{1}\right)=\theta_{\varphi_{0}}\left(x,y_{1}\right)$ and in
general $\varphi_{n+1}\left(x,y_{n+1}\right)=\theta_{\varphi_{n}}\left(x,y_{n+1}\right)$.
Fix some $n$ with $3 \leq n<\omega$. Let $\chi_{n}\left(y_{0},\ldots,y_{n-1},x_{0};z_{0},\ldots,z_{n-1};x_{0}'\right)$
with $\left|z_{i}\right|=\left|y_{i}\right|$ say that
$$
(\forall x)[\varphi_{i+1}\left(x,y_{i+1}\right)\to\varphi_{i}\left(x,z_{i}\right) ]
$$
for all $i<n-1$ and $\varphi_{n-1}\left(x_{0},y_{n-1}\right)\land\neg\varphi_{0}\left(x_{0}',y_{0}\right)$.
We will show that $\chi = \chi_{n}$ witnesses SOP$_{n}$ for all $n\geq3$. 

Let $\langle a_{t} : t<\omega\rangle$ be some infinite sequence in $\mathbb{M}$.
For $t < \omega, i<n$, let $b_{t}^{i}\in\mathbb{M}^{y_{i}}$ be such that $\varphi_{i}\left(a_{s},b_{t}^{i}\right)$
holds iff $s\leq t$ (i.e., witnessing that $\varphi_{i}$ has the
order property) and $(\forall x)[\varphi_{i+1}\left(x,b_{t}^{i+1}\right)\to\varphi_{i}\left(x,b_{t'}^{i}\right)]$
for all $t'\geq t$. We may find such $b_{t}^{i}$'s by induction
on $i<n$ using Condition \ref{cond:for singular compactness} and
compactness as above. For $k<\omega$, let $\bar{b}_{k}=b_{k}^{0}\ldots b_{k}^{n-1}a_{k}$.
We have that for $k,l<\omega$, $\mathbb{M}\models\chi\left(\bar{b}_{k},\bar{b}_{l}\right)$ if and only if $k < l$.
However, it is impossible that $\{\chi\left(\bar{x}_{k},\bar{x}_{k+1}\right):k<n-1\}\cup\left\{ \chi\left(\bar{x}_{n-1},\bar{x}_{0}\right)\right\} $
is consistent, since if it were realized by $\bar{c}_{k}=c_{k}^{0}\ldots c_{k}^{n-1}d_{k}$
for $k<n$, then $\varphi_{n-1}\left(d_{0},c_{0}^{n-1}\right)\Rightarrow \varphi_{n-2}\left(d_{0},c_{1}^{n-2}\right)\Rightarrow \cdots\Rightarrow\varphi_{0}\left(d_{0},c_{n-1}^{0}\right)$
but as $\chi\left(\bar{c}_{n-1},\bar{c}_{0}\right)$ holds, we have that $\neg\varphi_{0}\left(d_{0},c_{n-1}^{0}\right)$ holds as well which is a contradiction. 
\end{proof}
We give an example where this criterion holds. 
\begin{exmp}
\label{exa:local example}Let $L=\{P_{i}:i<3\}\cup\left\{ R_{0,1},R_{0,2},R_{1,2}\right\} $
where the $P_{i}$s are unary predicates and the $R_{i,j}$s are binary
relation symbols. Let $T^{\forall}$ say that $\langle P_{i}: i<3\rangle$
are disjoint and their union covers the universe, that $R_{i,j}\subseteq P_{i}\times P_{j}$
and that:
\begin{itemize}
\item [$\oslash$]If $R_{1,2}\left(b,c\right)$ then $(\forall x)\left[R_{0,1}\left(x,b\right)\to R_{0,2}\left(x,c\right)\right]$.
\end{itemize}
\end{exmp}

\begin{claim}
\label{claim:JEP,AP}$T^{\forall}$ is universal, it has the amalgamation
property (AP) and the joint embedding property (JEP). 
\end{claim}

\begin{proof}
The fact that $T^{\forall}$ is universal is clear. 

JEP: suppose that $M_{1},M_{2}\models T^{\forall}$ are disjoint.
Let $M$ be the following structure. As a set it is $M_{1}\cup M_{2}$.
For every relation symbol $Q\in L$, let $Q^{M}=Q^{M_{1}}\cup Q^{M_{2}}$. 

AP: suppose that $M_{0},M_{1},M_{2}\models T^{\forall}$ and $M_{0}\subseteq M_{1},M_{2}$
and $M_{0}=M_{1}\cap M_{2}$. Let $M$ be the following structure.
Its universe is just the union of the universes of $M_{1},M_{2}$.
For $i<3$, $P_{i}^{M}=P_{i}^{M_{1}}\cup P_{i}^{M_{2}}$. $R_{0,1}^{M}=R_{0,1}^{M_{1}}\cup R_{0,1}^{M_{2}}$
and similarly define $R_{1,2}^{M}=R_{1,2}^{M_{1}}\cup R_{1,2}^{M_{2}}$.
Let 
\begin{align*}
R_{0,2}^{M} & =R_{0,2}^{M_{1}}\cup R_{0,2}^{M_{2}}\\
 & \cup\{\left(a,b\right): a\in P_{0}^{M_{1}}\backslash M_{0},b\in P_{2}^{M_{2}}\backslash M_{0}\}\\
 & \cup\{\left(a,b\right) : a\in P_{0}^{M_{2}}\backslash M_{0},b\in P_{2}^{M_{1}}\backslash M_{0}\}.
\end{align*}
Let us check that $\oslash$ holds. Suppose that $M\models R_{1,2}\left(b,c\right)$.
Then we may assume that $b,c\in M_{1}$ (for $M_{2}$ it is the same
argument). Suppose that $M\models R_{0,1}\left(a,b\right)$. Then
if $a\in M_{1}$ then $M_{1}\models R_{0,2}\left(a,c\right)$. Otherwise
$a\in M_{2}$ and $b\in M_{0}$. If $c\in M_{0}$ as well, then $M_{2}\models R_{1,2}\left(b,c\right)\land R_{0,1}\left(a,b\right)$
so $M_{2}\models R_{0,2}\left(a,c\right)$ and we are done. Otherwise
$c\in M_{1}\backslash M_{0}$, in which case $R_{0,2}^{M}\left(a,c\right)$
holds by choice of $R_{0,2}^{M}$. 
\end{proof}
\begin{cor}
$T^{\forall}$ has a model completion
$T$ which has quantifier elimination. 
\end{cor}
%

\begin{prop}
$T$ is NSOP$_{4}$ and has SOP$_{3}$.
\end{prop}

\begin{proof}
We start by showing that $T$ is NSOP$_{4}$. Suppose that $\langle a_{i} : i<\omega \rangle$
is an indiscernible sequence in some model $M\models T$
which witnesses SOP$_{4}$. Let $A_{i}$ be $a_{i}$ as a set. Let
$M_{0}=A_{2}$, $M_{0}'=A_{3}$, $M_{1}=A_{1}A_{2}$, $M_{2}=A_{2}A_{3}$
and $M_{3}=A_{3}A_{4}$ with the induced structure from $M$. So all
are models of $T^{\forall}$. Let $M'$ be the amalgam of $M_{1},M_{2}$
over $M_{0}$ as defined in the proof of Claim \ref{claim:JEP,AP},
and similarly let $M''$ be the amalgam of $M_{2},M_{3}$ over $M_{0}'$.
Note that both $M'$ and $M''$ contain $M_{2}$ as a substructure
and that the universe of $M'$ is $A_{1}A_{2}A_{3}$ and of $M''$
is $A_{2}A_{3}A_{4}$, but neither are necessarily substructures of
$M$. 

Now we can amalgamate $M'$ and $M''$ over $M_{2}$. Moreover,
\begin{itemize}
\item [$\bullet$]Any structure $N$ whose universe is $A_{1}A_{2}A_{3}A_{4}$
which contains both $M'$, $M''$ as substructures and satisfies $T^{\forall}$
except perhaps $\oslash$, and such that $N\upharpoonright A_{1}A_{4}\models T^{\forall}$
will be a model of $T^{\forall}$ (i.e., $\oslash$ just follows). 
\end{itemize}
To see this, suppose that $N\models R_{1,2}\left(b,c\right)\land R_{0,1}\left(a,b\right)$.
We have to show that $N\models R_{0,2}\left(a,c\right)$. Note that
for every $x\in N$, if $x\in A_{i}\cap A_{j}$ for distinct $i,j\in\left\{ 1,\ldots,4\right\} $,
$x\in\bigcap_{i=1}^{4}A_{i}$ by indiscernibility. 

If $a,b,c$ all belong to either $A_{1}A_{2}A_{3}$, $A_{2}A_{3}A_{4}$
or $A_{1}A_{4}$ then this is clear, so assume this is not the case.

Suppose that $b,c\in A_{1}A_{2}A_{3}$, $a\in A_{4}$ (so $a\notin A_{1}A_{2}A_{3}$)
and $b\notin A_{1}$. Then if $b\in A_{2}\backslash A_{3}$ then $M''\models\neg R_{0,1}\left(a,b\right)$
--- contradiction, so $b\in A_{3}$. Then it must be that $c\in A_{1}\backslash A_{2}$
and $b\in A_{3}\backslash A_{2}$ so $M'\models\neg R_{1,2}\left(b,c\right)$
--- contradiction. 

If $b,c\in A_{1}A_{2}A_{3}$, $a\in A_{4}$ and $b\in A_{1}$ then
$c\notin A_{1}$. If $c\in A_{2}\backslash A_{3}$ then $M''\models R_{0,2}\left(a,c\right)$
so we are done. Else, $c\in A_{3} \setminus A_{2}$, so since $b\notin A_{2}$, $M'\models\neg R_{1,2}\left(b,c\right)$
--- contradiction. 

Suppose that $b\in A_{1}$ and $c\in A_{4}$. Then $a\in A_{2}A_{3}$.
If $a\in A_{2}\backslash A_{3}$ then $M''\models R_{0,2}\left(a,c\right)$
so we are done. Otherwise, $a\in A_{3}\backslash A_{2}$, so $M'\models\neg R_{0,1}\left(a,b\right)$
--- contradiction.

The case where $b\in A_{4}$ and $c\in A_{1}$ is done similarly. 

By symmetry, this covers all the cases so the bullet  is proved. 

Let $\sigma:A_{1}A_{4}\to A_{1}A_{4}$ be a bijection such that $\sigma(a_{1}) = a_{4}$ and $\sigma(a_{4}) = a_{1}$ as tuples (hence $\sigma^{2} = \mathrm{id}$).  Let $N_{0}$ be an amalgam of $M'$ and $M''$ over $M_{2}$ with domain $A_{1}A_{2}A_{3}A_{4}$.  Now define $N$ to be a structure with the same underlying set and the same interpretation of the unary predicates, but with each $R_{i,j}$ interpreted as follows:
$$
R^{N}_{i,j} = \left(R_{i,j}^{N_{0}} \setminus (A_{1}A_{4})^{2} \right) \cup \{(a,b) \in A_{1}A_{4} : M \models R_{i,j}(\sigma(a),\sigma(b))\}.  
$$
By indiscernibility, if $a,b$ are either both in $A_{1}$ or both in $A_{4}$, then $(a,b) \in R^{N}_{i,j}$ if and only if $(a,b) \in R_{i,j}^{M}$.  Then it is clear that $N$ has underlying set $A_{1}A_{2}A_{3}A_{4}$ and extends both $M'$ and $M''$, hence it satisfies the conditions in the bullet point above.  This shows $N\models T^{\forall}$,
and hence there is some $N'\models T$ containing $N$. 

But then, if $\varphi\left(x,y\right)$ is any quantifier-free formula
with $M\models\varphi\left(a_{1},a_{2}\right)$, then $N'\models\varphi\left(a_{1},a_{2}\right)\land\varphi\left(a_{2},a_{3}\right)\land\varphi\left(a_{3},a_{4}\right)\land\varphi\left(a_{4},a_{1}\right)$.
By quantifier elimination, $T$ is NSOP$_{4}$. 

Next we show that $T$ has SOP$_{3}$. For this we will use the following
criterion.
\begin{fact}
\cite[Claim 2.19]{shelah1996toward} For a theory $T$, having SOP$_{3}$
is equivalent to finding two formulas $\varphi\left(x,y\right),\psi\left(x,y\right)$
and a sequence $\langle a_{i},b_{i} : i<\omega \rangle$ in some $M\models T$
such that
\begin{itemize}
\item For all $i<j$, $M\models\neg\exists x\left(\varphi\left(x,a_{j}\right)\land\psi\left(x,a_{i}\right)\right)$. 
\item If $i\leq j$ then $M\models\varphi\left(b_{j},a_{i}\right)$ and
if $j<i$ then $M\models\psi\left(b_{j},a_{i}\right)$. 
\end{itemize}
\end{fact}
(The definition in \cite{shelah1996toward} additionally requires that $\{\varphi(x;y),\psi(x;y)\}$ is inconsistent, but this added condition is unnecessary:  given $\varphi$ and $\psi$ as above, one can replace $\varphi$ by $\varphi' = \varphi(x;y) \wedge \neg \psi(x;y)$ and then $\varphi'$ and $\psi$ will witness the above conditions).  

Let $\varphi\left(x,y'\right)=R_{0,1}\left(x,y'\right)$ and $\psi\left(x,y''\right)=\neg R_{0,2}\left(x,y''\right)$.
Let $\langle a_{i}',a_{i}'',b_{i} : i<\omega \rangle$ be a sequence such
that $R_{1,2}\left(a_{i}',a_{j}''\right)$ iff $i>j$, $R_{0,1}\left(b_{j},a_{i}'\right)$
whenever $i\leq j$ and $\neg R_{0,2}\left(b_{j},a_{i}''\right)$
whenever $i>j$. This sequence exists in some model $M\models T$
as we can define a model of $T^{\forall}$ which contains exactly
those elements. Now letting $a_{i}=\left(a_{i}',a_{i}''\right)$,
the first bullet follows from $\oslash$ and the second bullet by
the choice of $a_{i}',a_{i}''$ and $b_{i}$.
\end{proof}
\begin{cor}
There is a theory $T$ with NSOP$_{4}$ having SOP$_{3}$ such that
Condition \ref{cond:for singular compactness} holds with $\Delta=\left\{ R_{0,2}\left(x,y\right)\right\} $ and $\theta_{\varphi}$
from there being $R_{0,1}$. Thus $T$ has $\Delta$-singular compactness by Lemma \ref{lem:sufficient condition}. 
\end{cor}

\begin{proof}
We only need to show that Condition $\ref{cond:for singular compactness}$ holds. Suppose that $M\models T$ and $r$ is some finite $\Delta$-type.
Let $A\subseteq M$ be a finite set of realizations. Now the definition
of $T$, we may find some $b\in M$ with $R_{1,2}\left(b,c\right)$
whenever $R_{0,2}\left(x,c\right)\in r$ and $R_{0,1}\left(a,b\right)$
for all $a\in A$. This suffices. 
\end{proof}

\bibliographystyle{plain}
\bibliography{ms.bib}{}

\begin{thebibliography}{10}

\bibitem{baldwin1999transfering}
John~T Baldwin, Rami Grossberg, and Saharon Shelah.
\newblock Transfering saturation, the finite cover property, and stability.
\newblock {\em The Journal of Symbolic Logic}, 64(2):678--684, 1999.

\bibitem{ChernikovNTP2}
Artem Chernikov.
\newblock Theories without the tree property of the second kind.
\newblock {\em Ann. Pure Appl. Logic}, 165(2):695--723, 2014.

\bibitem{ArtemNick}
Artem Chernikov and Nicholas Ramsey.
\newblock On model-theoretic tree properties.
\newblock {\em Journal of Mathematical Logic}, page 1650009, 2016.

\bibitem{dvzamonja2004maximality}
Mirna D{\v{z}}amonja and Saharon Shelah.
\newblock On $\vartriangleleft^{*}$-maximality.
\newblock {\em Annals of Pure and Applied Logic}, 125(1):119--158, 2004.

\bibitem{kaplan2017some}
Itay Kaplan, Elad Levi, and Pierre Simon.
\newblock Some remarks on dp-minimal groups.
\newblock In {\em Groups, Modules, and Model Theory-Surveys and Recent
  Developments}, pages 359--372. Springer, 2017.

\bibitem{kaplan2017kim}
Itay Kaplan and Nicholas Ramsey.
\newblock On kim-independence.
\newblock {\em J. Eur. Math. Soc. (JEMS)}, 2017.
\newblock accepted, arXiv:1702.03894.

\bibitem{20-Kaplan2015}
Itay Kaplan, Saharon Shelah, and Pierre Simon.
\newblock Exact saturation in simple and {NIP} theories.
\newblock {\em J. Math. Log.}, 17(1):1750001, 18, 2017.

\bibitem{MR3666452}
Maryanthe Malliaris and Saharon Shelah.
\newblock Model-theoretic applications of cofinality spectrum problems.
\newblock {\em Israel J. Math.}, 220(2):947--1014, 2017.

\bibitem{ramsey2015invariants}
Nicholas Ramsey.
\newblock Invariants related to the tree property.
\newblock {\em arXiv preprint arXiv:1511.06453}, 2015.

\bibitem{shelah1980simple}
Saharon Shelah.
\newblock Simple unstable theories.
\newblock {\em Annals of Mathematical Logic}, 19(3):177--203, 1980.

\bibitem{shelah1990classification}
Saharon Shelah.
\newblock {\em Classification theory: and the number of non-isomorphic models}.
\newblock Elsevier, 1990.

\bibitem{shelah1996toward}
Saharon Shelah.
\newblock Toward classifying unstable theories.
\newblock {\em Annals of Pure and Applied Logic}, 80(3):229--255, 1996.

\bibitem{simon2015guide}
Pierre Simon.
\newblock {\em A guide to NIP theories}.
\newblock Cambridge University Press, 2015.

\end{thebibliography}

\end{document}